\documentclass[a4paper,12pt,fleqn]{amsart}
\usepackage[T1]{fontenc}
\usepackage[utf8]{inputenc}

\title{Lifting in compact covering spaces for fractional Sobolev mappings}

\author{Petru Mironescu}
\address{Univ Lyon\\
Universit\'e Claude Bernard Lyon 1\\
CNRS UMR 5208, Institut Camille Jordan
\\
F-69622 Villeurbanne, France}
\email{mironescu@math.univ-lyon1.fr}
\address{Simion Stoilow Institute of Mathematics of the Romanian Academy\\
 Calea Grivi\c tei 21\\
 010702 Bucure\c sti\\ 
 Rom\^ania
}

\author{Jean Van Schaftingen}
\address{Universit\'e catholique de Louvain\\ 
Institut de Recherche en Math\'ematique et Physique\\
Chemin du Cyclotron 2 bte L7.01.01\\
1348 Louvain-la-Neuve\\
Belgium}
\email{Jean.VanSchaftingen@uclouvain.be}

\thanks{This work has been initiated  during a long term visit  of P. Mironescu at the  Simion Stoilow Institute of Mathematics of the Romanian Academy; he thanks the Institute and the Centre Francophone en Math\'ematiques in Bucharest for their support on that occasion. 
J. Van Schaftingen was supported by the Mandat d'Impulsion Scientifique F.4523.17, ``Topological singularities of Sobolev maps'' of the Fonds de la Recherche Scientifique--FNRS. }

\date{July 2, 2019}
\usepackage{color}
 \definecolor{db}{rgb}{0.0,0.0,0.8} 
\definecolor{dg}{rgb}{0.0,0.55,0.14}
\definecolor{dr}{rgb}{0.5,0,0.07}

\usepackage[multiple]{footmisc}
\usepackage[hyperfootnotes=false,frenchlinks=true,colorlinks=true,linkcolor=db,citecolor=dg,
 bookmarks=true]{hyperref}
\hypersetup{pdftitle={Lifting in compact covering spaces for fractional Sobolev mappings},%
pdfauthor={Petru Mironescu and Jean Van Schaftingen},%
final}

\usepackage[T1]{fontenc}
\usepackage[hmargin=19mm,vmargin=28mm]{geometry}
\usepackage{microtype}
\usepackage{lmodern}

\usepackage{amssymb}
\usepackage{amsmath}
\usepackage{amsthm}
\usepackage{esint}

\usepackage[foot]{amsaddr}
\usepackage{paralist}

\usepackage[abbrev]{amsrefs}

\usepackage{csquotes}

\newtheorem{theorem}{Theorem}
\newtheorem{proposition}{Proposition}[section]
\newtheorem{lemma}[proposition]{Lemma}
\newtheorem{corollary}[proposition]{Corollary}

\theoremstyle{definition}

\theoremstyle{remark}
\newtheorem{remark}[proposition]{Remark}

\numberwithin{equation}{section}

\newcommand{\abs}[1]{{\lvert #1 \rvert}}

\newcommand{\seminorm}[1]{{\lvert #1 \rvert}}
\newcommand{\floor}[1]{{\lfloor #1 \rfloor}}

\newcommand{\dif}{\,\mathrm{d}}

\newcommand{\Rset}{\mathbb{R}}
\newcommand{\Zset}{\mathbb{Z}}

\newcommand{\Sset}{\mathbb{S}}
\newcommand{\Cset}{\mathbb{C}}
\newcommand{\Bset}{\mathbb{B}}

\newcommand{\compose}{\,\circ\,}

\DeclareMathOperator{\dist}{dist}
\DeclareMathOperator{\tr}{tr}
\DeclareMathOperator*{\osc}{osc}
\DeclareMathOperator*{\inter}{int}
\DeclareMathOperator{\diam}{diam}

\DeclareMathOperator{\Aut}{Aut}

\DeclareMathOperator{\inj}{inj}

\newcommand{\manifold}[1]{\mathcal{#1}}
\newcommand{\lifting}[1]{\widetilde{#1}}



\begin{document}
  
\begin{abstract}
  Let  $\pi : \widetilde{\mathcal{N}} \to \mathcal{N}$ be a Riemannian covering, with $\mathcal{N}$, $\widetilde{\mathcal{N}}$ smooth compact connected Riemannian manifolds. 
  If  $\mathcal{M}$ is an $m$--dimensional compact simply-connected Riemannian manifold,    $0<s<1$ and $2 \le sp< m$, we prove that every mapping $u \in W^{s, p} (\mathcal{M}, \mathcal{N})$ has a lifting  in $W^{s,p}$, i.e., we have $u = \pi \, \circ \, \widetilde{u}$ for some  mapping $\widetilde{u} \in W^{s, p} (\mathcal{M},  \widetilde{\mathcal{N}})$.  
  Combined with previous contributions of Bourgain, Brezis and Mironescu and Bethuel and Chiron, our result \emph{settles  completely} the question of the lifting in Sobolev spaces over covering spaces. 

The proof relies on an a priori estimate of the oscillations of $W^{s,p}$ maps with $0<s<1$ and $sp>1$, \emph{in dimension \(1\)}. Our argument also leads to the existence of a lifting when  $0<s<1$ and $1<sp<2\le m$, provided there is no topological obstruction on $u$, i.e., $u = \pi \, \circ \, \widetilde{u}$ holds in this range provided $u$ is in the strong closure of $C^\infty({\mathcal{M}},  \mathcal{N})$. 

However, when $0<s<1$, $sp = 1$ and $m\ge 2$, we show that an (analytical) obstruction still arises, even in absence  of topological obstructions. More specifically, we construct some map $u\in W^{s,p}(\mathcal{M},\mathcal{N})$ in the strong closure of $C^\infty({\mathcal{M}},  \mathcal{N})$, such that $u = \pi \, \circ \, \widetilde{u}$ does not hold for any $\widetilde{u} \in W^{s, p} ({\mathcal{M}},  \widetilde{\mathcal{N}} )$.
\end{abstract}

\maketitle

\section{Introduction}
Let \(\pi \in C^\infty (\lifting{\manifold{N}}, \manifold{N})\) be a   Riemannian covering. In most of the results we present, we make the following assumptions on the Riemannian manifolds $ \lifting{\manifold{N}}$, $\manifold{N}$ and on the cover $\pi$:
\begin{gather}
\label{a1}
\manifold{N}\text{ is compact and connected},
\\
\label{a2}
\lifting{\manifold{N}}\text{ is connected}
\end{gather}
and
\begin{equation}
\label{a3}
\pi\text{ is non-trivial}.
\end{equation}
In what follows, the compactness of $\lifting{\manifold{N}}$ will play a crucial role. We distinguish between the \emph{compact case} (when $\lifting{\manifold{N}}$ is compact) and the \emph{non-compact case}  (when $\lifting{\manifold{N}}$ is non-compact).

We also consider some \(\manifold{M}\) satisfying
\begin{equation}
 \label{a4}
 \begin{split}
&\manifold{M}\text{ is an \(m\)--dimensional compact simply-connected Riemannian manifold},\\
 &\text{possibly with boundary}.
 \end{split}
\end{equation}
In particular, we cover the case where \(\manifold{M}\) is a smooth bounded simply-connected domain in $\Rset^m$. (With a slight abuse, in this case we identify $\manifold{M}$ and $\overline{\manifold{M}}$.)
 
 \smallskip
A classical result in homotopy theory states that every map \(u \in C^k (\manifold{M}, \manifold{N})\)
can be lifted in $C^k$, i.e., there exists  some map \(\lifting{u} \in C^k (\manifold{M}, \lifting{\manifold{N}})\) such that 
\( u=\pi \compose \lifting{u}\) in \(\manifold{M}\). The \emph{lifting problem for Sobolev mappings} consists in determining whether 
every map \(u \in W^{s, p} (\manifold{M}, \manifold{N})\) can be lifted in $W^{s,p}$, i.e., whether there exists  some map \(\lifting{u} \in W^{s,p} (\manifold{M}, \lifting{\manifold{N}})\) such that that 
\(u=\pi \compose \lifting{u}\).

\medskip
We pause here to describe the Sobolev semi-norm we consider. Although we briefly consider the case where $s\ge 1$, in the new results we present we always assume that $0<s<1$ and $1\le p<\infty$. For such $s$ and $p$, the adapted semi-norm is defined as follows. We let  \(d_{\manifold{M}}\) and $d_{\manifold{N}}$ denote respectively the geodesic distances  on $\manifold{M}$ and $\manifold{N}$. We embed $\manifold{M}$ into some Euclidean space $\Rset^\mu$ and consider the $m$-dimensional Hausdorff measure on $\manifold{M}$, denoted $\mathrm{d} x$. We set 
\begin{equation*}
W^{s,p}(\manifold{M}, \manifold{N}):=\bigl\{ u: \manifold{M}\to \manifold{N};\, u\text{ is measurable and }|u|_{W^{s,p}}<\infty\bigr\},
\end{equation*}
where the Gagliardo semi-norm is defined as 
\begin{equation*}
  |u|_{W^{s,p}}^p:=\int_{\manifold{M}}\int_{\manifold{M}}\frac{d_{\manifold{N}}(u(x), u(y))^p}{d_{\manifold{M}} (x,y)^{m+sp}}\dif x \dif y.
\end{equation*}
Different embeddings of $\manifold{M}$ lead to the same space $W^{s,p}(\manifold{M}, \manifold{N})$, with equivalent semi-norms.

 In the case where the target manifold $\manifold{N}$ is compact, we can as well embed it into some Euclidean space $\Rset^\nu$, and then we may replace the geodesic distance by the Euclidean one. This leads to the same space, with equivalent semi-norm. 
 The space $W^{s, p} (\manifold{M}, \lifting{\manifold{N}})$ can be defined similarly; even when \(\manifold{N}\) is compact, the covering space \(\lifting{\manifold{N}}\) need not be compact.

\medskip

We next present some previous results on lifting. When \(\pi : \Rset \to \Sset^1\) is the universal covering of the circle by the real line, i.e., in complex notation,  we have \(\pi :  \Rset \ni \lifting{x} \mapsto e^{\imath\lifting{x}} \in \Sset^1 \subset \Cset\),
Bourgain, Brezis and Mironescu \cite{Bourgain_Brezis_Mironescu_2000} have showed that every map $u\in W^{s,p}(\manifold{M}, \Sset^1)$ has a lifting unless either \(1 \le sp < 2 \le m\) or [\(0 < s < 1\) and \(1 \le sp < m\)].  
Bethuel and Chiron \cite{Bethuel_Chiron_2007} have proved that the \emph{same conclusion} holds, more generally, in the  \emph{non-compact case}, under the assumptions \eqref{a1}--\eqref{a4}\footnote{In \cite{Bethuel_Chiron_2007}, \(\pi : \lifting{\manifold{N}} \to\manifold{N}\) is assumed to be the universal covering of $\manifold{N}$, but  the proofs there use only the assumptions \eqref{a1}--\eqref{a4}.}. The proof in \cite{Bethuel_Chiron_2007} relies, among other ingredients, on the  existence of a ray (i.e., an isometrically embedded real half-line) in any non-compact  connected Riemannian manifold. 
The \emph{compact case} was only partially settled in \cite{Bethuel_Chiron_2007}, one of the difficulties  in \cite{Bethuel_Chiron_2007} arising from the non-existence of rays in this case.  More specifically, the case where $0<s<1$ and $2\le sp<m$ was left open in \cite{Bethuel_Chiron_2007}. 

Our main result,  Theorem~\ref{theorem_intro_lifting} below,
completes their analysis\footnote{We exclude from our analysis the case where $m=1$, and thus $\manifold{M}$ is a bounded interval. In this case, the lifting property holds for any $s$ and $p$ \citelist{\cite{Bourgain_Brezis_Mironescu_2000}\cite{Bethuel_Chiron_2007}}.}.

\begin{theorem}
  \label{theorem_intro_lifting}
  Assume \eqref{a1}--\eqref{a4}, with $\lifting{\manifold{N}}$ compact and $m=\dim \manifold{M}\ge 2$.\\
  Then exactly one of the following holds. 
  \begin{enumerate}
  \renewcommand{\labelenumi}{{\rm{(\alph{enumi})}}}
\renewcommand{\theenumi}{\Alph{enumi}}
    \item Every map \(u \in W^{s, p} (\manifold{M}, \manifold{N})\) can be lifted into a map \(\lifting{u} \in W^{s, p} (\manifold{M}, \lifting{\manifold{N}})\).
   \item  \(1 \le sp < 2\).
  \end{enumerate}
\end{theorem}

The compact  case covers as important examples the real projective spaces $\Rset{\mathbb P}^m$, with universal covering space $\Sset^m$, which is relevant 
in the theory liquid crystals \citelist{\cite{Ball_Zarnescu_2011}\cite{Mucci_2010}} and  the \(d\)--fold covering of the circle, with $d\ge 2$, corresponding to \(\manifold{N}=\lifting{\manifold{N}}=\Sset^1\) and,  in complex notation\footnote{Strictly speaking, the metrics should be adapted by a constant conformal factor so that the mapping is a local Riemannian isometry.},  \( \pi (\lifting{x}) = \lifting{x}^d\). In this latter case, the lifting problem  is also known as the \(d\)th root problem. The solution of this problem is positive unless \(1 \le sp < 2 \le m\) \citelist{\cite{Bethuel_Chiron_2007}\cite{Mironescu_2010}}; 
the original proof of this fact is based  on the existence of liftings over the universal covering of \(\Rset\) by \(\Sset^1\) in the sum \((W^{s, p} + W^{1, sp}) (\manifold{M}, \Rset)\) \citelist{\cite{Mironescu2008}\cite{Mironescu_preprint}} and on the fractional Gagliardo--Nirenberg interpolation inequality \cite{gnp}. Our above result provides an alternative argument to the \(d\)th root problem. 

\smallskip
As noted by Bethuel \cite{bethueliff}, Theorem~\ref{theorem_intro_lifting} has as a consequence that, under the assumptions that 
\(p \ge 3\), the fundamental group \(\pi_1 (\manifold{N})\) is finite and the homotopy groups \(\pi_2 (\manifold{N}), \dotsc, \pi_{\floor{p - 1}} (\manifold{N})\) are trivial, then
the trace operator 
\begin{equation*}
W^{1,p}(\manifold{M}\times (0,1), \manifold{N})\ni f\mapsto \tr f\in W^{1-1/p,p}(\manifold{M}, \manifold{N})
\end{equation*}
is surjective. We will come back to this in a subsequent work  \cite{Mironescu_VanSchaftingen_Traces}.

\smallskip
Returning to the lifting question, 
it is instructive to compare the above picture with the one in the non-compact case, already completed in 
\cite{Bethuel_Chiron_2007}. 

\begin{theorem}[Bethuel and Chiron \cite{Bethuel_Chiron_2007}]
  \label{b1}
  Assume \eqref{a1}--\eqref{a4}, with $\lifting{\manifold{N}}$ non-compact and $m=\dim \manifold{M}\ge 2$.\\
Then exactly one of the following holds.
  \begin{enumerate}
  \renewcommand{\labelenumi}{{\rm{(\alph{enumi})}}}
\renewcommand{\theenumi}{\Alph{enumi}}
    \item Every map \(u \in W^{s, p} (\manifold{M}, \manifold{N})\) can be lifted into a map \(\lifting{u} \in W^{s, p} (\manifold{M}, \lifting{\manifold{N}})\).
   \item  \(1 \le sp < 2\) or $[0<s<1\text{ and }1\le sp<\dim \manifold{M}]$.
  \end{enumerate}
\end{theorem}

Theorem~\ref{b1} contains as a special case  the result established in  \cite{Bourgain_Brezis_Mironescu_2000} for \(\pi : \Rset \to \Sset^1\) the universal covering of the unit circle.

\smallskip
The proof  of Theorem~\ref{theorem_intro_lifting} relies on a new one-dimensional estimate, \eqref{b2} below, that may be of independent interest. For the sake of simplicity, we state it for real-valued continuous functions  $f\in C^0(\Rset , \Rset)$. For such $f$ and $x, y\in\Rset$, we define the oscillation of \(f\) on the interval \([x, y]\) as 
\begin{equation}
\label{b23}
{\textstyle \osc_{[x, y]} f} :=\max\{ |f(z)-f(t)|;\, z, t \in [x, y]\}.
\end{equation}

We prove that, for $0<s<1$ and $1<p<\infty$ such that $sp>1$, we have the \emph{reverse oscillation inequality}
\begin{equation}
\label{b2}
\int_\Rset\int_\Rset \frac{[\osc_{[x, y]} f]^p}{|y-x|^{1+sp}}\dif x \dif y\le C_{s,p}\int_\Rset\int_\Rset \frac{|f(y)-f(x)|^p}{|y-x|^{1+sp}}\dif x \dif y.
\end{equation}
The terminology \enquote{reverse inequality} refers to the fact that, since  $\osc_{[x, y]} f \ge |f(y)-f(x)|$, we  have, for any $0<s<1$ and $1\le p<\infty$, 
\begin{equation}
\label{b3}
\int_\Rset\int_\Rset \frac{|f(y)-f(x)|^p}{|y-x|^{1+sp}}\dif x \dif y\le 
\int_\Rset\int_\Rset \frac{[\osc_{[x, y]} f]^p}{|y-x|^{1+sp}}\dif x \dif y.
\end{equation}
Our result \eqref{b2} is that the inequality \eqref{b3} can be reversed when $sp>1$.

\medskip
We next turn to the nature of obstructions to the existence of lifting. They are of two types, topological and analytical ones. 
Topological obstructions arise when \(1 \le sp < 2\le m\), and are induced by maps which are locally of the form \(u (y, z) = f (y/\abs{y})\), where \((y, z) \in \Bset^2 \times \Bset^{m - 2}\) and the map \(f \in C^0 (\Sset^1, \manifold{N})\) admits no lifting. (Here  and in the sequel, $\Bset^k$ denotes the unit ball of $\Rset^k$.) The existence of such $f$ follows from our assumption \eqref{a3}.
Analytical obstructions arise when \(0 < s < 1\) and \(1 \le sp < m\); they are related to the existence of maps $\lifting{u}: \Bset^m\to  \lifting{\manifold{N}}$ that are smooth except at the origin, such that roughly speaking $\pi\compose \lifting{u}$ oscillates much less than $\lifting{u}$, i.e., $\lifting{u}\in W^{s,p}_{loc}(\Bset^m\setminus\{0\}, \lifting{\manifold{N}})\setminus W^{s,p}(\Bset^m, \lifting{\manifold{N}})$,  
while $\pi\compose \lifting{u}\in W^{s,p}(\Bset^m, \manifold{N})$.  


\smallskip
Theorem~\ref{theorem_intro_lifting} has a variant which is valid when $1<sp<2$. Indeed,  the maps that include topological obstructions are not in the strong closure of $C^\infty({\manifold M}, \manifold N)$ for the $W^{s,p}$ norm (this can be seen by a simple topological argument \cite[Lemma 1 and Appendix A.2]{Bethuel_Chiron_2007}).  With this in mind, Theorem~\ref{b4} below asserts that, in absence of topological obstructions, there are no analytical obstructions. 
\begin{theorem}
\label{b4}
Assume \eqref{a1}--\eqref{a4}, with $\lifting{\manifold{N}}$ compact. Assume that $0<s<1$ and $1<sp<2\le m=\dim\manifold{M}$. 
Consider,  for a map \(u \in W^{s, p} (\manifold{M}, \manifold{N})\), the following properties:
  \begin{enumerate}
  \renewcommand{\labelenumi}{{\rm{(\alph{enumi})}}}
\renewcommand{\theenumi}{\alph{enumi}}
    \item \label{it_Ue5ef_a}$u$ can be strongly approximated by maps in $C^\infty({\manifold M}, \manifold N)$,
     \item \label{it_Ue5ef_b}
  $u$ can be weakly approximated by maps in $C^\infty({\manifold M}, \manifold N)$,
  \item \label{it_Ue5ef_c}
  $u$ has a lifting in $W^{s, p} (\manifold{M}, \lifting{\manifold{N}})$.
  \end{enumerate}
  Then
  \begin{enumerate}
   \renewcommand{\labelenumi}{{\rm{(\roman{enumi})}}}
\renewcommand{\theenumi}{\roman{enumi}}
  \item
    We have $\eqref{it_Ue5ef_a} \implies \eqref{it_Ue5ef_b} \implies \eqref{it_Ue5ef_c}$.
  \item
  If $\manifold{M}$ is diffeomorphic to a ball and \(\pi\) is the universal covering, then the properties \eqref{it_Ue5ef_a}, \eqref{it_Ue5ef_b} and \eqref{it_Ue5ef_c} are all equivalent.  \end{enumerate}
\end{theorem}

We specify the notion of strong convergence in Theorem~\ref{b4}, since there is no natural distance on $W^{s,p}(\manifold{M}, {\manifold{N}})$. 
We embed the manifold $\manifold{N}$ into some Euclidean space $\Rset^\nu$, and thus identify $W^{s,p}(\manifold{M}, \manifold{N})$ with $W:=\{ v\in W^{s,p}(\manifold{M}, \Rset^\nu);\, v(x)\in \manifold{N}\text{ for a.e. \(x \in \manifold{M}\)}\}$. With this identification, $u_j\to u$ in $W^{s,p}(\manifold{M}, \manifold{N})$ amounts to $u_j, u \in W$ and $u_j\to u$ in $W^{s,p}(\manifold M, \Rset^\nu)$ as \(j \to \infty\). When $\manifold{N}$ is compact or, more generally, when the sequence $(u_j)_{j \ge 0}$ takes its values into a fixed compact subset of $\manifold{N}$, this notion of convergence does not depend on the embedding.

We also specify the notion of weak convergence, since $W^{s,p}(\manifold{M}, {\manifold{N}})$ is not a linear space. 
When $0<s<1$ and $1<p<\infty$, we adopt the following convention: $u_j\to u$ weakly in $W^{s,p}(\manifold{M}, \manifold{N})$ if $u_j\to u$ a.e.\ as $j\to\infty$ and $|u_j|_{W^{s,p}(\manifold{M})}\le C$, $\forall\, j$.

\smallskip
It will be clear from its proof  that Theorem~\ref{b4} is still valid when $s=1$ and $p\ge 1$. In the case of the universal covering of $\Sset^1$, the conclusion of the theorem still holds when $s>1$ \cite[Chapters 9 and 11]{bmbook}. When $s>1$ and for a general covering, the definition of $W^{s,p}(\manifold{M}, \lifting{\manifold{N}})$ is less obvious. Adopting the definition of $W^{s,p}(\manifold{M}, \lifting{\manifold{N}})$ in \cite{Bethuel_Chiron_2007}, Theorem~\ref{b4} with $s>1$  can possibly be obtained by combining \cite[Appendix A.1]{Bethuel_Chiron_2007} with the composition result in \cite{gnp}; this is not investigated here.

\medskip
Theorem~\ref{b4} leaves open the question of existence of analytical obstructions when $0<s<1$ and  \(sp = 1\). Such obstructions \emph{do exist}, as shows our next result.

\begin{theorem}
\label{b5}
Assume \eqref{a2}--\eqref{a4}, $\manifold{N}$ connected and $m=\dim \manifold{M}\ge 2$. For $0<s<1$ and  \(p\) such that \(sp = 1\) and for every point \(a\in {\manifold{M}}\), there exists a mapping $u : \manifold{M} \to \manifold{N}$ such that 
\begin{enumerate}
 \renewcommand{\labelenumi}{{\rm{(\roman{enumi})}}}
\renewcommand{\theenumi}{\roman{enumi}}
\item
\(u \in C^\infty ({\manifold{M}}\setminus \{a\}, \manifold{N}) \cap W^{s, p}(\manifold{M}, \manifold{N})\),
  \item $u$ can be strongly approximated by maps in $C^\infty({\manifold M}, \manifold N)$,
    \item $u$ has no lifting \(\lifting{u} \in W^{s, p} (\manifold{M}, \lifting{\manifold{N}})\).
\end{enumerate}
\end{theorem}

Theorem~\ref{b5} answers negatively \cite{Mironescu_2010}*{open problem 7}.

\medskip
Our paper is organized as follows. In Section~\ref{s2} we recall some basic facts about coverings. 
In Section~\ref{s5}, which is the main contribution of this work, we prove the reverse oscillation inequality \eqref{b2} and its consequences, Theorems~\ref{theorem_intro_lifting} and \ref{b4}. 
In Section~\ref{s3} we discuss uniqueness, in a framework more general than the one of the universal covering of the circle \cite{Bourgain_Brezis_Mironescu_2000} or of universal coverings \cite{Bethuel_Chiron_2007}. This will be needed in the proof of the existence of the analytic obstruction. In Section~\ref{s4}, we prove Theorem~\ref{b5}.

\section{About coverings}
\label{s2}

Let us start by recalling some basic fact concerning the coverings.
The mapping \(\pi : {\lifting{\manifold{N}}} \to \manifold{N}\) (with $\manifold{N}$, $\lifting{\manifold{N}}$ topological spaces) is a \emph{cover} (or \emph{covering map}) whenever 
 \(\pi\) is continuous and every point \(y \in \manifold{N}\) belongs to an open set \(U \subset \manifold{N}\) \emph{evenly covered} by $\pi$, i.e., 
the inverse image \(\pi^{-1} (U)\) is a disjoint union of open sets $V_i$, $i\in I$, with \(\pi:V_i\to U\) a homeomorphism, $\forall\, i\in I$. 

If \(\manifold{N}\) is a connected topological manifold and if the covering space \({\lifting{\manifold{N}}}\) is connected, then the cardinality of the inverse image \(\pi^{-1} (\{y\})\) of a point does not depend on the point \(y \in \manifold{N}\) and is at most countable; this follows from the fact that \(\pi^{-1} (\{y\})\) is isomorphic to \(\pi_1 (\manifold{N}, y)\) \cite{Hatcher_2002}*{Proposition 1.32} combined with the fact that $\pi_1(\manifold{N}, y)$ is at most countable, $\forall\, y\in  \manifold{N}$ \cite{Lee_2011}*{Theorem 7.21}.

If \(\manifold{N}\) is a connected Riemannian manifold, then the cover  \(\pi\) induces on  \({\lifting{\manifold{N}}}\)  a unique \emph{Riemannian structure} such that the mapping \(\pi\) is a local isometry. 
Conversely, if the Riemannian manifold \({\lifting{\manifold{N}}}\) is complete and 
if the mapping \(\pi : {\lifting{\manifold{N}}} \to \manifold{N}\) is a local isometry 
(that is, the pullback \(\pi^* g\) of  the metric \(g\) of \(\manifold{N}\) coincides with the metric \(\lifting{g}\) of \(\lifting{\manifold{N}}\)), 
then  \(\pi\) is a cover \cite[Lemma 11.6]{Lee_1997}.
The local isometry property implies in particular that $\pi$  is globally a non-expansive map: 
for every \(\lifting{x}, \lifting{y} \in {\lifting{\manifold{N}}}\), we have
\[
  d_{\manifold{N}} (\pi (\lifting{x}), \pi (\lifting{y})) \le d_{{\lifting{\manifold{N}}}} (\lifting{x}, \lifting{y}),
\]
with equality everywhere if and only if the map \(\pi\) is a global homeomorphism.

\medskip
The next lemma shows that a Riemannian covering map is always an isometry on scales smaller than the injectivity radius \(\inj (\manifold{N})\)
(which is defined as the least upper bound of the radii \(\rho>0\) such that the exponential mapping at any point $y\in \manifold{N}$, restricted to a ball of radius \(\rho\) of the tangent space $T_y \manifold{N}$, is a diffeomorphism).

\begin{lemma}
\label{lemma_small_isometry}
Let \(\pi : {\lifting{\manifold{N}}} \to \manifold{N} \) be a Riemannian covering map. Assume that   \(\manifold{N}\) has positive injectivity radius \(\inj (\manifold{N}) > 0\).\\
Then for every \(\lifting{x}, \lifting{y} \in {\lifting{\manifold{N}}}\) such that \(d_{{\lifting{\manifold{N}}}} (\lifting{x}, \lifting{y}) \le \inj (\manifold{N})\), one has \( d_{{\lifting{\manifold{N}}}}(\lifting{x}, \lifting{y}) = d_{\manifold{N}}(\pi (\lifting{x}), \pi (\lifting{y}))\).
\end{lemma}

The positivity assumption on the injectivity radius in Lemma~\ref{lemma_small_isometry} is satisfied in particular when the manifold $\manifold{N}$ is compact.

The proof of Lemma~\ref{lemma_small_isometry} follows the strategy to prove that local isometries of complete manifolds yield covering maps \cite[proof of Lemma 11.6]{Lee_1997}.

\begin{proof}%
[Proof of  Lemma~\ref{lemma_small_isometry}]
Let \(\lifting{x}, \lifting{y} \in {\lifting{\manifold{N}}}\) satisfy  \(d_{{\lifting{\manifold{N}}}} (\lifting{x}, \lifting{y}) \le \inj (\manifold{N})\). 
Let \(\lifting{\gamma} : [0, 1] \to \lifting{\manifold{N}}\) be the natural parametrization of a  minimizing geodesic $\lifting\Gamma$  in \(\lifting{\manifold{N}}\) joining the point \(\lifting{x}\) to \(\lifting{y}\).
Since  \(\pi\) is a local isometry, \(\gamma := \pi \circ \lifting{\gamma} : [0, 1] \to \manifold{N}\) is the natural parametrization of a geodesic $\Gamma$  in \(\manifold{N}\) joining the point \(x:=\pi (\lifting{x})\) to \(y:=\pi (\lifting{y})\). Moreover, the  length   of $\Gamma$ is \( d_{{\lifting{\manifold{N}}}} (\lifting{x}, \lifting{y}) \le \inj (\manifold{N})\). By definition of the injectivity radius, this geodesic is minimal, and thus \( d_{{\lifting{\manifold{N}}}}(\lifting{x}, \lifting{y}) = d_{\manifold{N}}(\pi (\lifting{x}), \pi (\lifting{y}))\).
\end{proof}

If \(\pi : {\lifting{\manifold{N}}} \to \manifold{N}\) is a cover, its \emph{group of deck transformations} is the set
\begin{equation*}
  \Aut (\pi) 
  = \bigl\{ \tau : {\lifting{\manifold{N}}} \to {\lifting{\manifold{N}}};\, \tau \text{ is a homeomorphism and } \pi \compose \tau = \pi \bigr\}.
\end{equation*}
The set \(\Aut (\pi)\) is a group under the composition operation and is also known as the \emph{Galois group} of the cover \(\pi\). 
Assuming $\lifting{\manifold{N}}$ to be connected and $\lifting x_0\in\lifting{\manifold{N}}$, an element $\tau \in \Aut (\pi)$ is uniquely determined by $\tau(\lifting x_0)$. Therefore, if \(\manifold{N}\) is a connected topological manifold and if \({\lifting{\manifold{N}}}\) is connected, then \(\Aut(\pi)\) is at most countable.
If \(\pi\) is a Riemannian covering, then the elements of the group \(\Aut (\pi)\) are \emph{global} isometries of the manifold \(\lifting{\manifold{N}}\).

As examples of groups of deck transformations, if $\pi:\Rset\to\Sset^1$ is  the universal covering of $\Sset^1$,  then \(\Aut (\pi)\) is the group of translations of \(\Rset\) by integer multiples of \(2 \pi\) and is isomorphic to \(\Zset\), and if \(\pi:\Sset^m\to\Rset{\mathbb P}^m\) is the universal covering of the projective space $\Rset{\mathbb P}^m$,  then \(\Aut (\pi) = \{\operatorname{id}, - \operatorname{id}\}\), which is isomorphic to \(\Zset_2\).


\smallskip
A covering \(\pi\) is \emph{normal} whenever the action of \(\Aut (\pi)\) is transitive on the fibers of \(\pi\), that is, whenever,  given \(\lifting{x}, \lifting{y} \in {\lifting{\manifold{N}}}\) such that \(\pi (\lifting{x}) = \pi (\lifting{y})\), there exists an automorphism \(\tau \in \Aut (\pi)\) such that \(\lifting{y} = \tau (\lifting{x})\). Normal coverings are also known as \emph{regular coverings} or as \emph{Galois coverings}. An important case of normal covering is the universal covering of a connected Riemannian manifold \cite[Proposition 1.39]{Hatcher_2002}. 

\section{Lifting}
\label{s5}
\subsection{Proof of the reverse oscillation inequality \eqref{b2}} We consider some continuous function $f\in C^0(\manifold{I}, \Rset)$, with $\manifold{I}=(a,b)\subseteq \Rset$ some interval. Then \eqref{b2} holds on $\manifold{I}$, for some constant independent of $\manifold{I}$ and $f$. In order to prove \eqref{b2}, we start from the Morrey embedding $W^{\sigma, p}(\manifold{J})\hookrightarrow C^{0, \sigma-1/p}(\manifold{J})$, valid for any interval $\manifold{J}=(z,t)\subseteq\Rset$ and for $1/p<\sigma<1$. In a quantitative form, this embedding implies that, with a constant $C$ depending only on $\sigma$ and $p$, we have 
\begin{equation}
  \label{ha1}
  |g(t)-g(z)|\le C\, (t-z)^{\sigma-1/p}\, |g|_{W^{\sigma,p}((z,t))},\ \forall\, g\in C^0([z,t]),\ \forall\, {-\infty}<z<t<\infty .
\end{equation}
(For an elementary proof of this well-known property, see e.g.\ \cite[Lemma 3]{m_hardy}.)
In turn, \eqref{ha1} implies that 
\begin{equation}
  \label{ha2}
  {\textstyle \osc_{[x, y]} f} \le C\, (y-x)^{\sigma-1/p}\, |f|_{W^{\sigma,p}((x,y))},\ \forall\, f\in C^0(\manifold{I}),\ \forall\, a<x <y<b.
\end{equation}
We next choose some $\sigma$ such that $1/p<\sigma<s$ (this is possible, since $sp>1$) and find, via \eqref{ha2}, that\footnote{In what follows, $A\lesssim B$ stands for $A\le  C B$, with $C$ an absolute constant.}
\begin{equation*}
  \begin{split}
    \int_{\manifold{I}}\int_{\manifold{I}} \frac{[\osc_{[x, y]} f]^p}{|y-x|^{1+sp}}\dif x \dif y\lesssim &\iint_{a<x<y<b} \frac{|f|_{W^{\sigma, p}((x,y))}^p}{(y-x)^{2+(s-\sigma)p}}\dif x \dif y\\
    \lesssim & \iiiint_{a<x<t<z<y<b}\frac{|f(z)-f(t)|^p}{(z-t)^{1+\sigma p}}\frac 1{(y-x)^{2+(s-\sigma)p}}\dif x \dif y\dif z \dif t \\
    \le &\iint_{a<t<z<b}\frac{|f(z)-f(t)|^p}{(z-t)^{1+\sigma p}}\\
    &\hskip 18mm\times \left(\iint_{-\infty<x<t<z<y<\infty}\frac 1{(y-x)^{2+(s-\sigma)p}}\dif x \dif y\right)\, \dif z \dif t \\
    \lesssim &\iint_{a<z<t<b}\frac{|f(z)-f(t)|^p}{(z-t)^{1+\sigma p}}\frac{1}{(z-t)^{(s-\sigma) p}}\dif t \dif z=\frac 12 |f|_{W^{s,p}(\manifold{I})}^p,
  \end{split}
\end{equation*}
whence \eqref{b2}. \hfill$\square$

\medskip
In the same spirit, we have the following estimate for maps with values into manifolds. Let $\manifold{I}=(a,b)\subseteq \Rset$ and $u\in C^0(\manifold{I}, \manifold{N})$, where $\manifold{N}$ is a connected Riemannian manifold. By analogy with  \eqref{b23}, we define the oscillation
\begin{equation}
  \label{ia1}
  {\textstyle \osc_{[x, y]}} u :=\max\{ d_{\manifold{N}}(u(z),u(t));\, z, t \in [x, y] \}.
\end{equation}

\begin{lemma}
  \label{ia2}
  Let $0<s<1$ and $1<p<\infty$ be such that $sp>1$.\\
  Let $\manifold{N}$ be a connected Riemannian manifold.\\
  Let $\manifold{I}=(a,b)\subseteq \Rset$ and  $u\in C^0(\manifold{I}, \manifold{N})$.\\
  Then
  \begin{equation}
    \label{ia3}
    \int_{\manifold{I}}\int_{\manifold{I}}\frac{[\osc_{[x, y]} u]^p}{|y-x|^{1+sp}}\dif x \dif y\le C_{s,p}\, |u|_{W^{s,p}(\manifold{I})}^p=C_{s,p}\int_{\manifold{I}}\int_{\manifold{I}}\frac{d_{\manifold{N}}(u(x), u(y))^p}{|y-x|^{1+sp}}\dif x \dif y.
  \end{equation}
\end{lemma}
\begin{proof} Write $\manifold{I}=(a,b)$ and let $a<z<t<b$. 
  Applying \eqref{ha1} with $g(\alpha):=d_{\manifold{N}} (u(\alpha), u(z))$, $\forall\, \alpha\in [z,t]$, and using the inequality $|g(\alpha)-g(\beta)|\le d_{\manifold{N}} (u(\alpha), u(\beta))$, $\forall\, \alpha, \beta\in [z,t]$, we find that
  \begin{equation}
    \label{ib1}
    |d_{\manifold{N}}(u(t),u(z))|\le C\, (t-z)^{\sigma-1/p}\, |u|_{W^{\sigma,p}((z,t))},\  \forall\, a<z<t<b,
  \end{equation}
  and thus 
  \begin{equation}
    \label{ib2}
    {\textstyle \osc_{[x, y]} u} \le C\, (y-x)^{\sigma-1/p}\, |u|_{W^{\sigma,p}((x,y))},\  \forall\, a<x <y<b.
  \end{equation}
  
  We then continue as in the proof of \eqref{b2}.
\end{proof}

\subsection{The one-dimensional estimate for lifting}
\label{s5.2}
We assume  here that
\begin{gather}
  \label{ii1}
  0<s<1\text{ and }1<p<\infty\text{ are such that }sp>1,
  \\
  \label{ii2}
  \pi\in C^\infty(\lifting{\manifold{N}}, \manifold{N})\text{ is a Riemannian covering and }\lifting{\manifold{N}}\text{ is compact}.
\end{gather} 

Let us note that \eqref{ii2} implies that \(\manifold{N}\) is compact and thus $0<\inj(\manifold{N})<\infty$, and that $\diam(\lifting{\manifold{N}})<\infty$.

Let $\manifold{I}=(a,b)\subseteq \Rset$ and $u\in C^0(\manifold{I}, \manifold{N})$. Then we may lift $u$ as $u=\pi\compose\lifting{u}$, for some $\lifting{u}\in C^0(\manifold{I}, \lifting{\manifold{N}})$, uniquely determined by its value at some point of $\manifold{I}$.

\begin{lemma}
  \label{ha3}
  Assume \eqref{ii1}--\eqref{ii2}.\\
  Let $\manifold{I}\subseteq\Rset$ be an interval and $u\in C^0(\manifold{I}, \manifold{N})$.\\
  Then every continuous lifting $\lifting{u}\in C^0(\manifold{I}, \lifting{\manifold{N}})$ of $u$ satisfies
  \begin{equation}
    \label{ha4}
    |\lifting{u}|_{W^{s,p}(\manifold{I})}^p\le C_{s,p}\, \left(\frac{\diam(\lifting{\manifold{N}}) }{\inj(\manifold{N})}\right)^p\, |u|_{W^{s,p}(\manifold{I})}^p,
  \end{equation}
  for some absolute constant $C_{s,p}$. 
\end{lemma}
\begin{proof}
  Let $\manifold{I}=(a,b)$. We have the obvious estimate
  \begin{equation}
    \label{ic1}
    d_{\lifting{\manifold{N}}}(\lifting{u}(x),\lifting{u}(y))\le \diam (\lifting{\manifold{N}}),\ \forall\, x, y\in \manifold{I}.
  \end{equation}
On the other hand, if  $x, y\in\manifold{I}$ and $\osc_{[x, y]} u \le \inj (\manifold{N})$,
  then \(d_{\lifting{\manifold{N}}}(\lifting{u}(x), \lifting{u}(y))\le \inj (\manifold{N})\) and thus, by Lemma~\ref{lemma_small_isometry},
  \begin{equation}
    \label{id1}
    d_{\lifting{\manifold{N}}}(\lifting{u}(x), \lifting{u}(y))\le {\textstyle \osc_{[x, y]}} u .
  \end{equation}
  Combining \eqref{ic1} with the conditional inequality \eqref{id1}, and noting that $\diam(\lifting{\manifold{N}})\ge \inj (\manifold{N})$, we find that
  \begin{equation}
    \label{id2}
    d_{\lifting{\manifold{N}}}(\lifting{u}(x), \lifting{u}(y))\le \frac{\diam(\lifting{\manifold{N}}) }{\inj(\manifold{N})}
    {\textstyle \osc_{[x, y]} u},\ \forall\, x, y\in \manifold{I}.
  \end{equation}
  We obtain \eqref{ha4} from Lemma~\ref{ia2} and \eqref{id2}.
\end{proof}

\begin{remark}
  The estimate  \eqref{ha4} \emph{has to depend}  on \(\diam (\lifting{\manifold{N}}) /\inj (\manifold{N})\). Indeed, consider the $d$-fold covering $\pi_d$ of $\Sset^1$,  with $d\ge 1$. In this case, we have  $\lifting{\manifold{N}}=d\, \Sset^1$, \(\pi_d (d \cos t, d \sin t) = (\cos (d t), \sin (d t))\), $\forall\, t\in\Rset$,  $\inj (\manifold{N})=\pi$, $\diam(\lifting{\manifold{N}})=\pi\dif$. 
  Let $\xi\in\Rset$. If we set \(u_{d, \xi} (x) := d (\cos (\xi  x), \sin (\xi  x))\in\lifting{\manifold{N}}\), $\forall\, x\in (0,1)$, then we have 
  \(\pi_d (u_{d, \xi}) = u_{1, d \xi}\).
  On the other hand, we have, with $0<C<\infty$ some absolute constant,
  \[
  \lim_{\abs{\xi} \to \infty} \frac{\seminorm{u_{d, \xi}}_{W^{s, p} ((0,1))}^p}{d^p |\xi|^{sp-1}}
  = C,\ \lim_{\abs{\xi} \to \infty} \frac{\seminorm{\pi_d\circ u_{d, \xi}}_{W^{s, p} ((0,1))}^p}{(d |\xi|)^{sp-1}}=C,
  \]
  and thus 
  \begin{equation}
    \label{ib5}
    \lim_{\abs{\xi} \to \infty} \frac{\seminorm{u_{d, \xi}}_{W^{s, p} ((0,1))}^p}{\seminorm{\pi_d\circ u_{d, \xi}}_{W^{s, p} ((0,1))}^p}=d^{p-sp+1}= \left(\frac{\diam(\lifting{\manifold{N}}) }{\inj(\manifold{N})}\right)^{p-sp+1}. 
  \end{equation}
  
  Note, however, that the estimates \eqref{ha4} and \eqref{ib5} do not yield the same power of $\diam(\lifting{\manifold{N}}) /\inj(\manifold{N})$. 
  The question about the optimal power in \eqref{ha4} is open.
\end{remark}

\subsection{The dimensional reduction argument }
\label{s5.3}
In this section and the next one, we explain how to derive $m$-dimensional estimates from the one-dimensional estimate provided by Lemma~\ref{ha3}. To start with, we consider the case of a cube, which is very simple. The case of a general domain requires slightly  more work and is presented in the next section.

\begin{lemma}
  \label{ie1}
Assume \eqref{ii1}--\eqref{ii2}.\\
Let $\manifold{C}:=a + (0,\ell)^m$, with \(\ell \in (0, \infty)\) and \(a \in \Rset^m\). 
Let $\manifold{Q}\subset \manifold{C}$ be an open set such that
  \begin{enumerate}
   \renewcommand{\labelenumi}{{\rm{(\roman{enumi})}}}
\renewcommand{\theenumi}{\roman{enumi}}
    \item \label{it_Aivie7_1}
    $\manifold{Q}$ is simply-connected,
    \item \label{it_Aivie7_2}
    for every $i=1,\ldots, m$ and for a.e.\ $\widehat x_i:=(x_1,\ldots, x_{i-1}, x_{i+1},\ldots, x_m)\in (0,\ell)^{m-1}$, we have $a+(x_1,\ldots, x_{i-1}, t, x_{i+1},\ldots, x_m)\in \manifold{Q}$, $\forall\, t\in (0,\ell)$.
  \end{enumerate}
Let $u:\manifold{C}\to \manifold{N}$ be such that $u\in C^0(\manifold{Q}, \manifold{N})$. 
\\  
Then every continuous lifting $\lifting{u}\in C^0(\manifold{Q}, \lifting{\manifold{N}})$ of $u$ satisfies
  \begin{equation}
    \label{if5}
    |\lifting{u}|_{W^{s,p}(\manifold{C})}^p\le C_{s,p,m}\, \left(\frac{\diam(\lifting{\manifold{N}}) }{\inj(\manifold{N})}\right)^p\, |u|_{W^{s,p}(\manifold{C})}^p,
  \end{equation}
  for some absolute constant $C_{s,p,m}$.
\end{lemma}

The existence of the lifting $\lifting{u}$ follows from assumption \eqref{it_Aivie7_1} on $\manifold{Q}$.
By assumption \eqref{it_Aivie7_2} on $\manifold{Q}$, $\manifold{C}\setminus\manifold{Q}$ is a null set, and thus $\lifting{u}$ is defined a.e.\ on $\manifold{C}$.

\begin{proof}[Proof of Lemma~\ref{ie1}]
 With no loss of generality, we assume that $a=0$.  For $i=1,\ldots, m$ and $\widehat x_i\in (0,\ell)^{m-1}$, set 
  \begin{equation*}
    u_{\widehat x_i}(t):=u(x_1,\ldots, x_{i-1}, t, x_{i+1},\ldots, x_m),\ \forall\, t\in (0,\ell). 
  \end{equation*}
  By assumption \eqref{it_Aivie7_2}, $u_{\widehat x_i}$ is well-defined on $(0,\ell)$, for  $\widehat x_i$  in the complement of a null subset of $(0,\ell)^{m-1}$,  and for such $\widehat x_i$ we  define similarly $\lifting{u}_{\widehat x_i}(t)$.  By Lemma~\ref{ha3}, we have
  \begin{equation}
    \label{if3}
    \sum_{i=1}^m \int_{(0,\ell)^{m-1}} |\lifting{u}_{\widehat x_i}|_{W^{s,p}((0,\ell))}^p\dif\widehat{x}_i\le C_{s,p}\left(\frac{\diam(\lifting{\manifold{N}}) }{\inj(\manifold{N})}\right)^p\sum_{i=1}^m \int_{(0,\ell)^{m-1}} |u_{\widehat x_i}|_{W^{s,p}((0,\ell))}^p \dif\widehat{x}_i. 
  \end{equation} 
  We conclude by combining \eqref{if3} with the $\ell$-independent semi-norm equivalence 
  \begin{equation}
    \label{if4}
    \sum_{i=1}^m \int_{(0,\ell)^{m-1}} |f_{\widehat x_i}|_{W^{s,p}((0,\ell))}^p\dif\widehat{x}_i\sim 
    |f|_{W^{s,p}(\manifold{C})}^p,\ \forall\,  f:\manifold{C}\to \manifold{N}\end{equation}
  (and the similar equivalence for $\lifting{\manifold{N}}$-valued maps). For $\Rset$-valued maps defined on $\Rset^m$, this equivalence is well-known, see e.g.\ \cite[Lemma 7.44]{adams}. The argument for manifold-valued maps defined on  a cube is exactly the same as the one in \cite[proof of Lemma 7.44]{adams}. The fact that the constant $C_{s,p,m}$ does not depend on $\ell$ follows by scaling.
\end{proof}

\subsection{From local to global estimates}
\label{s5.4}

Here, we explain how to pass from \emph{local estimates} (on cubes) to \emph{global estimates} (on general domains).  The basic ingredient is the semi-norm control provided by the next result. \begin{lemma}
  \label{ik1}
  Let $0<s<1$ and $1\le p<\infty$.
  \\
  Let $\manifold{N}$ be a compact Riemannian manifold.
  \\
  Let $\manifold{M}$  be a connected compact manifold, possibly with boundary. 
  \\
  Let $(\manifold{C}_j)_{j\in J}$ be a finite family of open subsets of $\manifold{M}$, covering $\manifold{M}$.
  \\
  Then
  \begin{equation}
    \label{ik2}
    |u|_{W^{s,p}(\manifold{M})}^p\le C_{s, p, \manifold{M}}\sum_{j\in J}|u|_{W^{s,p}(\manifold{C}_j)}^p,\ \forall\,  u:\manifold{M}\to \manifold{N}.
  \end{equation} 
\end{lemma}

\begin{proof} Let $m$ be the dimension of $\manifold{M}$. 
  Let $\delta>0$ be such that 
  \begin{equation*}
    [x, y\in \manifold{M},\, d_{\manifold{M}}(x,y)<\delta]\implies [x, y\in {\manifold C}_j\text{ for some }j\in J].
  \end{equation*}
  The existence of $\delta$ implies that 
  \begin{equation}
    \label{im1}
    \begin{split}
      |u|_{W^{s,p}(\manifold{M})}^p\le &\sum_{j\in J}|u|_{W^{s,p}(\manifold{C}_j)}^p +\iint_{x, y\in\manifold{M},\, d_{\manifold{M}}(x, y) \ge \delta}\frac{d_\manifold{N} (u(x), u(y))^p}{d_{\manifold{M}}(x,y)^{m+sp}}\dif x \dif y\\
      \lesssim & \sum_{j\in J}|u|_{W^{s,p}(\manifold{C}_j)}^p +\iint_{x, y\in\manifold{M}} d_{\manifold{N}} (u(x),u(y))^p \dif x \dif y,
    \end{split}
  \end{equation}
  and thus \eqref{ik2} amounts to proving, the Poincar\'e type estimate
  \begin{equation}
    \label{im2}
    \iint_{x, y\in\manifold{M}} d_{\manifold{N}} (u(x),u(y))^p \dif x \dif y\lesssim \sum_{j\in J}|u|_{W^{s,p}(\manifold{C}_j)}^p.
  \end{equation}
 
 We may assume that every 
 $\mathcal{C}_j$ is non-empty.  Since \(\mathcal{M}\) is connected, we can relabel the sets \((\mathcal{C}_j)_{j \in I}\) as \((\mathcal{C}_j)_{1 \le j \le k}\)
  in such a way that \(\manifold{C}_{i + 1} \cap \bigcup_{j = 1}^i \manifold{C}_j  \ne \emptyset,\ \forall\, 1\le i\le k-1\).
  We then have, by the triangle inequality,
  for every \(x \in \bigcup_{j = 1}^{i} \manifold{C}_j\) and \(y \in \manifold{C}_{i + 1}\),
  \[
  d_{\manifold{M}} (u (x), u (y))^p
  \lesssim \fint_{\manifold{C}_{i + 1} \cap \bigcup_{j = 1}^{i} \manifold{C}_j} [d (u (x), u (z))^p + d (u (z), u (y))^p] \dif z,
  \]
and hence, by induction, we obtain
  \[
  \begin{split}
    \iint_{x, y \in \bigcup_{j = 1}^{i + 1} \manifold{C}_j} \hspace{-1.5em} d (u (x), u (y))^p  \dif x \dif y
    &\lesssim \int_{x, y \in \bigcup_{j = 1}^{i} \manifold{C}_j} \hspace{-1.5em}  d (u (x), u (y))^p\dif x \dif y
    + \int_{x, y \in \manifold{C}_{i+1}} \hspace{-1em} d (u (x), u (y))^p\dif x \dif y\\
& \lesssim \sum_{j = 1}^{i + 1} |u|_{W^{s,p}(\manifold{C}_j)}^p.\hskip 70mm\qedhere
\end{split}
  \]
\end{proof}

Combining Lemma~\ref{ie1} with  Lemma~\ref{ik1}, we obtain the following

\begin{corollary}
  \label{ij1}
  Assume \eqref{ii1}--\eqref{ii2}.\\
  Let $\manifold{M}\subset\Rset^m$ be a  smooth bounded open set. Let $\manifold{M}'\subset\Rset^m$ be an    open set such that $\overline{\manifold{M}}\subset \manifold{M}'$. \\  
  Let $\manifold{R}\subset \manifold{M}'$ and $u:\manifold{M}'\to \manifold{N}$ be  such that
  \begin{enumerate}
   \renewcommand{\labelenumi}{{\rm{(\roman{enumi})}}}
\renewcommand{\theenumi}{\roman{enumi}}
    \item \label{it_Aihoe2_1}
    for every cube $\manifold{C}\subset 
    \manifold{M}'$, the set $\manifold{Q}:=\manifold{R}\cap \manifold{C}$
    satisfies assumption \eqref{it_Aivie7_2} in Lemma~\ref{ie1},
    \item \label{it_Aihoe2_2} 
    $u\in C^0(\manifold{R}, \manifold{N})$ and $u$ has a lifting $\lifting{u}\in C^0(\manifold{R}, \lifting{\manifold{N}})$.
  \end{enumerate}
Then
  \begin{equation}
    \label{if9}
    |\lifting{u}|_{W^{s,p}(\manifold{M})}^p\le C_{s,p,\manifold{M}}\, \left(\frac{\diam(\lifting{\manifold{N}}) }{\inj(\manifold{N})}\right)^p\, |u|_{W^{s,p}(\manifold{M}')}^p,
  \end{equation}
  for some absolute constant $C_{s,p,\manifold{M}}$.
\end{corollary}

\subsection{Proof of Theorem~\ref{b4}}
\label{s5.5}
Since, clearly, $\eqref{it_Ue5ef_a} \implies \eqref{it_Ue5ef_b}$, it suffices to prove that $\eqref{it_Ue5ef_b} \implies\eqref{it_Ue5ef_c}$ (always)  and $\eqref{it_Ue5ef_c} \implies \eqref{it_Ue5ef_a}$ (in the case of the universal covering, with $\manifold{M}$ a ball).

\begin{proof}[Proof of $\eqref{it_Ue5ef_b} \implies \eqref{it_Ue5ef_c}$] We work on a compact manifold $\manifold{M}$. 
  In order to obtain \eqref{it_Ue5ef_c}, it suffices to obtain the following \emph{a priori estimate}. If $u\in C^\infty ({\manifold{M}}, \manifold{N})$, then $u$ has a lifting $\lifting{u}\in C^\infty ({\manifold{M}}, \lifting{\manifold{N}})$ such that
  \begin{equation}
    \label{in1}
    |\lifting{u}|_{W^{s,p}(\manifold{M})}^p \lesssim |u|_{W^{s,p}(\manifold{M})}^p.
  \end{equation}
  
  Indeed, assuming that \eqref{in1} holds for smooth maps, a straightforward limiting procedure shows that \eqref{in1} still holds for weak limits of smooth maps.
  
  In order to prove \eqref{in1}, we consider a finite covering of $\manifold{M}$ with open sets ${\manifold{C}}_j$, each one bi-Lipschitz homeomorphic to a cube in $\Rset^m$. On each ${\manifold{C}}_j$, we have 
  \begin{equation}
  \label{ra1}
  |\lifting{u}|_{W^{s,p}({\manifold{C}}_j)}^p\lesssim |{u}|_{W^{s,p}({\manifold{C}}_j)}^p;
  \end{equation}
  this follows (after composition with a suitable homeomorphism) 
   from Lemma \ref{ie1}. 
   
   We conclude using \eqref{ra1} and Lemma \ref{ik1} (applied to $\lifting{u}$). 
  \end{proof}

\begin{proof}[Proof of $\eqref{it_Ue5ef_c}\implies \eqref{it_Ue5ef_a}$] We work on an  open ball.  Write $u=\pi\compose\lifting{u}$, with $\lifting{u}\in W^{s,p}(\manifold{M}, \lifting{\manifold{N}})$. Since $1<sp<2$ and  $\lifting{\manifold{N}}$ is compact and simply-connected (by definition of the universal covering),  $C^\infty (\overline{\manifold{M}}, \lifting{\manifold{N}})$ is dense in $W^{s,p}(\manifold{M}, \lifting{\manifold{N}})$ \cite[Theorem 4]{Brezis_Mironescu_2015} 
  (see also \citelist{\cite{Bousquet_Ponce_Van_Schaftingen_2014}*{Theorem 1.3}\cite{Mucci_2009}*{Theorem 2}}). Consider a sequence $(\lifting{u}_n)_{n\ge 0}\) in \(C^\infty(\overline{\manifold{M}}, \lifting{\manifold{N}})$ such that $\lifting{u}_n\to \lifting{u}$ in $W^{s,p}(\manifold{M})$ as $n\to\infty$. Set $u_n:=\pi\compose  \lifting{u}_n \in C^\infty(\overline{\manifold{M}}, \manifold{N})$. Using the fact that $\pi$ is Lipschitz-continuous, we find that $u_n\to u$ in $W^{s,p}(\manifold{M})$ as $n\to\infty$.
\end{proof}

\begin{remark}
  \label{io1}
  We have proved the following quantitative version of \eqref{it_Ue5ef_c}. If $u\in W^{s,p}(\manifold{M}, \manifold{N})$ has a lifting $\lifting{u}\in W^{s,p}(\manifold{M}, \lifting{\manifold{N}})$, then
  \begin{equation*}
    |\lifting{u}|_{W^{s,p}(\manifold{M})}^p\le C_{s, p, \manifold{M}} \left(\frac{\diam(\lifting{\manifold{N}}) }{\inj(\manifold{N})}\right)^p\, |u|_{W^{s,p}(\manifold{M})}^p.
  \end{equation*}
\end{remark}

\subsection{Proof of Theorem~\ref{theorem_intro_lifting}}
\label{s5.6}
In view of the partial results of Bethuel and Chiron \cite{Bethuel_Chiron_2007}, it suffices to consider the case where $0<s<1$, $2\le sp<m=\dim\manifold{M}$. 

\medskip
\noindent
\emph{Proof of Theorem~\ref{theorem_intro_lifting} when  \(\manifold{M}\) is a smooth bounded domain of \(\Rset^m\)}.
As in the previous section, it suffices to prove the \emph{a priori estimate} 
\begin{equation}
  \label{ip1}
  |\lifting{u}|_{W^{s,p}(\manifold{M})}^p\le C_{s, p, \manifold{M}}\, |u|_{W^{s,p}(\manifold{M})}^p,
\end{equation}
for a lifting $\lifting{u}$ of $u$, where $u$ belongs to a dense subset of $W^{s,p}(\manifold{M}, \manifold{N})$. Weak density would suffice, but it turns out that we have at our disposal a convenient strongly dense class. Such a class is obtained as follows \cite{Brezis_Mironescu_2015}*{Theorem 6}. Extend first every $u\in W^{s,p}(\manifold{M}, \manifold{N})$  by reflection across $\partial \manifold{M}$ to a larger set $\manifold{M}'$. 
The extension, still denoted $u$, satisfies $u:\manifold{M}'\to\manifold{N}$ and
\begin{equation}
  \label{ip2}
  |u|_{W^{s,p}(\manifold{M}')}^p\le C_{s,p,\manifold{M}}\, |u|_{W^{s,p}(\manifold{M})}^p.
\end{equation}
Since \(\mathcal{M}\) is smooth, bounded and simply-connected, we can assume without loss of generality that \(\manifold{M}'\) is also  smooth, bounded and simply-connected.

Let $j:=\lfloor sp\rfloor$ denote the integer part of $sp$, so that $2\le j<m$.  Consider the $\varepsilon$-grids ${\manifold T}_{a, \varepsilon}$,  $\forall\, \varepsilon>0$,   $\forall\, a\in\Rset^m$, defined by the cubes $\manifold{C}_{a,\varepsilon, k}:=a+\varepsilon k+[0,\varepsilon]^m$, $k\in\Zset^m$. Let ${\manifold T}_{a, \varepsilon}^j$ denote the $j$th skeleton of ${\manifold T}_{a, \varepsilon}$ and ${\manifold U}_{a, \varepsilon}^{m-j-1}$ denote the ($(m-j-1)$-dimensional) dual skeleton of ${\manifold T}_{a, \varepsilon}^j$.

We use the following approximation result \cite{Brezis_Mironescu_2015}*{Theorem 6}: given $u\in W^{s,p}(\manifold{M}', \manifold{N})$, there exist sequences $\varepsilon_n\searrow 0$, $(a_n)_{n\ge 0}\subset\Rset^m$, $(u_n)_{n\ge 0}\subset W^{s,p}(\manifold{M}', \manifold{N})$ such that
\begin{enumerate}
 \renewcommand{\labelenumi}{{\rm{(\alph{enumi})}}}
\renewcommand{\theenumi}{\alph{enumi}}
  \item
  \label{aabove}
  $u_n\to u$ as $n\to\infty$, strongly in $W^{s,p}(\manifold{M}')$.
  \item
  $u_n$ is continuous in $\overline{\manifold{M}'}\setminus {\manifold U}_{a_n, \varepsilon_n}^{m-j-1}$, $\forall\, n\ge 0$.
\end{enumerate}

In view of item \eqref{aabove} above and of Corollary~\ref{ij1}, in order to obtain \eqref{ip1} (and thus to complete the proof of Theorem~\ref{theorem_intro_lifting}) it suffices to prove that $u_n$ and the set $\manifold{R}=\manifold{R}_n:=\manifold{M}'\setminus {\manifold U}_{a_n, \varepsilon_n}^{m-j-1}$ satisfy the assumptions \eqref{it_Aihoe2_1} and \eqref{it_Aihoe2_2} in Corollary~\ref{ij1}. 

Clearly, assumption \eqref{it_Aihoe2_1} is satisfied, since ${\manifold U}_{a_n, \varepsilon_n}^{m-j-1}$ is a finite union of $(m-j-1)$-dimensional affine subspaces and since $j\ge 1$.
Moreover, by a straightforward induction argument relying on the next lemma (which is a particular case of general position arguments), the set \(\mathcal{R}_n\) is simply-connected, and thus \({u_n}_{|{\mathcal{R}_n}}\) has a lifting \(\lifting{u}_n \in C^0 (\manifold{R}_n, \lifting{\manifold{N}})\). \qed

\begin{lemma}
\label{ra2} Let $m\ge 3$. 
  Let \(\manifold{V} \subset \Rset^m\) be open and let \(\Sigma\) be an affine subspace of dimension  \(n\le m - 3\).
  If \(\manifold{V}\) is simply-connected, then \(\manifold{V} \setminus \Sigma\) is simply-connected.
\end{lemma}
\begin{proof}[Proof of Lemma \ref{ra2}]
  Without loss of generality, we assume that \(0 \in \Sigma\). Let \(\gamma \in C^1 (\Sset^1, \manifold{V} \setminus \Sigma)\). Our aim is to prove that $\gamma$ is null homotopic in $\manifold{V} \setminus \Sigma$.
  
Since the set \(\manifold{V}\) is simply-connected, there exists \(\sigma \in C^1 (\overline{\Bset^2}, \manifold{V})\) such that 
\(\sigma_ {\vert {\Sset^1}} = \gamma\). 
Since the set \(\manifold{V}\) is open, there exists \(\delta > 0\), such that, for every \(x \in \Sset^1\), \(B_{\delta} (\gamma (x)) \subset \manifold{V} \setminus \Sigma\) and, for every \(x \in \overline{\Bset^2}\), \(B_{\delta} (\sigma (x)) \subset \manifold{V}\).\footnote{$B_\varepsilon (x)$ is the Euclidean ball of centre $x$ and radius $\varepsilon$, with $x\in\Rset^m$. When $x=0$, we write $B_\varepsilon$ instead of $B_\varepsilon (0)$.} 
Let \(P : \Rset^m \to \Sigma^\perp\) be the orthogonal projection on \(\Sigma^\perp\). Since  \(\dim \Sigma^\perp \ge  3\), \(P (\sigma (\overline{\Bset^2}))\) is a negligible subset of \(\Sigma^\perp\). Hence, for almost every 
$\xi \in B_{\delta}\cap \Sigma^\perp$, we have $-\xi\not\in P (\sigma (\overline{\Bset^2}))$. For any such $\xi$, we have  \((\sigma + \xi) (\overline{\Bset^2})\subset \manifold{V}\setminus \Sigma\), and thus $\gamma+\xi:\Sset^1\to \manifold{V}\setminus \Sigma$ is null homotopic in $\manifold{V} \setminus \Sigma$. We conclude by noting that, by the choice of \(\delta\), the maps \(\gamma:{\Sset^1}\to \manifold{V}\setminus \Sigma\) and \(\gamma+\xi:{\Sset^1}\to \manifold{V}\setminus \Sigma\) are homotopic in $\manifold{V}\setminus \Sigma$.

By a similar argument, if $\manifold{V}$ is connected, then $\manifold{V}\setminus \Sigma$ is connected.
\end{proof}

\begin{proof}[Proof of Theorem~\ref{theorem_intro_lifting} when  \(\manifold{M}\) is a compact manifold without boundary] We  embed  \(\manifold{M}\) isometrically into some Euclidean  space \(\Rset^\mu\). Then there exists $\delta>0$ such that:
\begin{enumerate}
 \renewcommand{\labelenumi}{{\rm{(\alph{enumi})}}}
\renewcommand{\theenumi}{\alph{enumi}}
  \item \label{it_Om0oo_1}
the 
nearest point projection \(\Pi : \manifold{O} \to \manifold{M}\) is well-defined and smooth on the set $\manifold{O}:=\{ x\in\Rset^\mu;\, \dist (x, 
\manifold{M})<\delta\}$;
\item \label{it_Om0oo_2}
$\manifold{O}$ is smooth;
\item \label{it_Om0oo_3}
for every $x\in\manifold{M}$, $\Pi^{-1}(\{ x\})$ is diffeomorphic to $\Bset^{\mu-m}$;
\item \label{it_Om0oo_4}
if $u:\manifold{M}\to\manifold{N}$ and we set $U:=u\circ \Pi:\manifold{O}\to\manifold{N}$, then 
\begin{equation}
\label{ra3}
C'|u|_{W^{s,p}(\manifold{M})}^p\le |U|_{W^{s,p}(\manifold{O})}^p\le C|u|_{W^{s,p}(\manifold{M})}^p
\end{equation} 
for some $C', C\in (0,\infty)$ depending on $0<s<1$, $1\le p<\infty$, the embedding, $\delta$, but independent of $u$.
\end{enumerate}
Let $u\in W^{s,p}(\manifold{M}, \manifold{N})$, and let $\manifold{O}$, $U$ as above. Then $\manifold{O}$ is simply-connected, since $\Pi:\manifold{O}\to \manifold{M}$ is a retraction  and $\manifold{M}$ is simply-connected.

By the first part of the proof of the theorem, there exists a map \(\lifting{U} \in W^{s, p} (\manifold{O}, \lifting{\manifold{N}})\) such that \(\pi \compose \lifting{U} = U\) in \(\manifold{O}\). 
Moreover, for a.e.\  \(x \in \manifold{M}\), \(U _{|\Pi^{-1} (\{x\})}\) is constant on \(\Pi^{-1} (\{x\})\)  (that we identify with a ball, see \eqref{it_Om0oo_3} above) and \(\lifting{U}_{|{\Pi^{-1} (\{x\})}} \in W^{s, p} (\pi^{-1} (\{x\}),\manifold{N})\).  Set $b:=u(x)$ and let $\pi^{-1}(b)=\{ \lifting{b}_i;\, i\in I\}$, so that $U(y)\in \{ \lifting{b}_i;\, i\in I\}$, for a.e.\ $y\in \Pi^{-1}(\{ x\})$. Consider some $i\in I$ such that the set $\{ y\in\Pi^{-1}(\{ x\});\, U(y)=\lifting{b}_i\}$ is non-negligible (such an $i$ does exist, since $I$ is at most countable). Since $\pi\circ U=\pi\circ \lifting{b}_i$ on $\Pi^{-1}(\{x\})$,  Proposition \ref{de1} below implies that $U=\lifting{b}_i$ a.e.\ on $\Pi^{-1}(\{x\})$. For any $x$ as above, set
$\lifting{u}(x):=\lifting{b}_i$, so that $\lifting{u}$ is defined a.e.\ on $\manifold{M}$ and $\lifting{U}\circ\Pi=\lifting{u}$. By \eqref{ra3}, we have $\lifting{u}\in W^{s,p}(\manifold{M}, \lifting{\manifold{N}})$ and, clearly, $\pi\circ \lifting{u}=u$.
\end{proof}

\begin{proof}[Proof of Theorem~\ref{theorem_intro_lifting} when  \(\manifold{M}\) is a compact manifold with boundary] This is a slightly more subtle case. 
We consider two larger smooth  compact   manifolds  with boundary, \(\manifold{M}'\) and $\manifold{M}''$,  such that \(\manifold{M}\subset \inter\, ({\manifold{M}'})\),  \(\manifold{M}'\subset \inter\, ({\manifold{M}''})\) (where $\inter$ stands for the interior),   and we can extend maps from $\manifold{M}$ to $\manifold{M}'$ by reflection across the boundary such that \eqref{ip2} holds.

We next embed  \(\manifold{M}''\) isometrically into some  \(\Rset^\mu\). Let $\Pi$ denote the nearest point projection on $\manifold{M}''$. Then, for small $\delta>0$, if we set $\manifold{O}:=\{ x\in\Rset^\mu;\, \dist (x, \manifold{M})<\delta\text{ and }\Pi(x)\in\manifold{M}\}$, then $\manifold{O}$ satisfies \eqref{it_Om0oo_1}, \eqref{it_Om0oo_3} and \eqref{it_Om0oo_4}, above, but not \eqref{it_Om0oo_2}. Thus we cannot directly apply directly \cite{Brezis_Mironescu_2015}*{Theorem 6} to the map $U$ in $\manifold{O}$ as above. However, we note that in order to invoke this result, we do not need a smooth domain. It suffice to know that there exists an open set $\manifold{O}'$ such that $\overline{\manifold{O}}\subset\manifold{O}'$ and an extension $V\in W^{s,p}(\manifold{O}', \manifold{N})$ of $U$. In our case, we let (again, for sufficiently small $\delta>0$) $\manifold{O}':=\{ x\in\Rset^\mu;\, \, \dist (x, \manifold{M}')<2\delta\text{ and }\Pi(x)\in\manifold{M}'\}$. The extension $V$ of $U$ to $\manifold{O}'$ is defined as follows. Let  $\overline u$ be the extension of $u$ to $\manifold{M}'$ by reflection across $\partial\manifold{M}$. Then we set,  in $\manifold{O}'$, $V:=\overline u\circ \Pi$. Clearly, $V$ has the required properties. We continue the proof as in the case of compact manifolds without boundary.

The proof of Theorem~\ref{theorem_intro_lifting}
is complete.
\end{proof}

\section{Uniqueness of Sobolev liftings}
\label{s3}

The role of this section is to provide tools for checking that analytical obstructions are indeed obstructions. Roughly speaking, the question we address  here is the following. Assume that $u:\manifold{M}\to\manifold{N}$ has \emph{some} \enquote{bad} lifting $\lifting{u}$. How to make sure that \emph{all other} possible liftings are also \enquote{bad}?

We present two types of results. The former ones (Proposition~\ref{proposition_uniqueness}, Proposition~\ref{proposition_uniqueness_normal},  Corollary~\ref{db1}) are valid in particular in the case of the universal coverings of compact connected manifolds. The latter ones (Proposition~\ref{de1}, Corollary~\ref{de2}) are valid  for more general coverings, but require more assumptions on the bad lifting. Although, strictly speaking, it is possible to prove Theorem~\ref{b5} using only Corollary~\ref{de2}, we find instructive to provide two different proofs, relying on different topological assumptions and analytical arguments.

Throughout this section, we make the following assumptions.
\begin{gather}
\label{dc1}
\pi \in C^\infty ({\lifting{\manifold{N}}}, \manifold{N})\text{ is a Riemannian covering},
\\
\label{dc2}
\lifting{\manifold{N}}\text{ and }\manifold{N}\text{ are  connected}, 
\\
\label{dc3}
\begin{aligned}
&\manifold{M}\text{ is a relatively compact connected open subset of some $m$-dimensional} \\
&\text{Riemannian manifold }\manifold{M}'.
\end{aligned}
\end{gather}

This includes as special cases the interior of a smooth compact manifold and bounded open sets in $\Rset^m$. (However,  if we restrict to open sets in $\Rset^m$, boundedness is not essential.) Our assumption on $\manifold{M}$ emphasizes the fact that the smoothness of the boundary of $\manifold{M}$ plays no role here. 

A subset of $\manifold{M}$ is negligible if it is, near each point and in local coordinates, the image of a negligible set for the $m$-dimensional Lebegue measure. 

\medskip 

The uniqueness results are obtained under the assumption
\begin{equation}
\label{dc4}
sp\ge 1, 
\end{equation}
which is the relevant one for uniqueness \cite{Bourgain_Brezis_Mironescu_2000}. In view of the applications we have in mind, we also assume that 
\begin{equation}
\label{dc5}
0<s<1,
\end{equation}
but this latter assumption in not necessary for the validity of the results below.

\medskip
Uniqueness being a local matter, we consider maps in $W^{s,p}_{loc}(\manifold{M})$. By a standard argument, it then suffices to prove uniqueness for maps in $W^{s,p}(B)$, with $B$ a ball in $\Rset^m$.

\begin{proposition}
\label{proposition_uniqueness}
Assume \eqref{dc1}--\eqref{dc5} and,  in addition  $\inj (\manifold{N})>0$. \\
Let \(\lifting{u}, \lifting{v} \in W^{s, p}_{loc} (\manifold{M}, \lifting{\manifold{N}})\) be such that \(\pi \compose \lifting{u} = \pi \compose \lifting{v}\)  on \(\manifold{M}\). Then either \(\lifting{u} = \lifting{v}\) a.e.\  on \(\manifold{M}\) or \(\lifting{u} \ne \lifting{v}\) a.e.\ on \(\manifold{M}\).
\end{proposition}

\begin{proof}As explained above, we may assume that $\manifold{M}$ is a ball and $\lifting{u}, \lifting{v} \in W^{s, p} (\manifold{M}, \lifting{\manifold{N}})$.  

Let us note that, if $\varphi:[0,\infty)\to\Rset$ is an $L$-Lipschitz function, then 
\begin{equation}
\label{dd1}
f:\manifold{M}\to\Rset, \ f(x):=\varphi( d_{\lifting{\manifold{N}}}(\lifting{u}(x), \lifting{v}(x))),\ \forall\, x\in \manifold{M}, 
\end{equation}
satisfies
\begin{equation*}
|f(x)-f(y)|\le L\, |d_{\lifting{\manifold{N}}}(\lifting{u}(x), \lifting{v}(x))-d_{\lifting{\manifold{N}}}(\lifting{u}(y), \lifting{v}(y))|\le L\, [d_{\lifting{\manifold{N}}}(\lifting{u}(x), \lifting{u}(y))+d_{\lifting{\manifold{N}}}(\lifting{v}(x), \lifting{v}(y))],
\end{equation*}
and thus $f\in W^{s,p}(\manifold{M}, \Rset)$.

Set $\ell:=\min\{1, \inj(\manifold{N})\}$ and   \(
  \varphi : [0, \infty) \to \Rset
  \),  $\varphi(t):=\min\{ t/\ell, 1\}$, $\forall\,  t\ge 0$.  The assumption $\pi\compose\lifting{u}=\pi\compose\lifting{v}$  implies, via Lemma~\ref{lemma_small_isometry},  that the corresponding function $f$ in \eqref{dd1} satisfies \(f (\manifold{M}) \subseteq \{0, 1\}\). Under the assumptions $sp\ge 1$ and $\manifold{M}$ connected, the space $W^{s,p}(\manifold{M}, \{ 0, 1\})$ contains only constant a.e.\ functions  \cite{Bourgain_Brezis_Mironescu_2000}*{theorem B.1} (see also \citelist{\cite{Bourgain_Brezis_Mironescu_2001}\cite{Brezis_2002}\cite{Bethuel_Demengel_1995}*{lemma A.1}\cite{Hardt_Kinderlehrer_Lin_1990}*{lemma 1.1}}). Thus either $f=0$ a.e.\ on \(\manifold{M}\), or $f=1$ a.e.\ on \(\manifold{M}\), whence the conclusion. 
\end{proof}

\begin{proposition}%
\label{proposition_uniqueness_normal}
Assume \eqref{dc1}--\eqref{dc5} and,  in addition, that   $\inj (\manifold{N})>0$ and that $\pi$ is a normal covering.\\
If \(\lifting{u}, \lifting{v} \in W^{s, p}_{loc} (\manifold{M}, \lifting{\manifold{N}})\) and if \(\pi \compose \lifting{u} = \pi \compose \lifting{v}\)  on \(\manifold{M}\), then there exists \(\tau \in \Aut (\pi)\) such that \(\lifting{v} = \tau \compose \lifting{u}\) a.e.\ on \(\manifold{M}\).
\end{proposition}

In the case where \(\pi\) is the universal covering of a compact connected Riemannian manifold, Proposition~\ref{proposition_uniqueness_normal} is due to Bethuel and Chiron \cite[Lemma A.4]{Bethuel_Chiron_2007}.
\begin{proof}%
  [Proof of Proposition~\ref{proposition_uniqueness_normal}]
For each deck transformation \(\tau \in \Aut (\pi)\), we define the measurable set 
\[
 A_{\tau} := \left\{ x \in \manifold{M};\, \lifting{v}(x) = \tau \compose \lifting{u}(x)\right\}.
\]
Since the covering \(\pi\) is normal, we have 
\[
    \manifold{M} 
  = 
    \bigcup_{\tau \in \Aut (\pi)} A_\tau.
\]
Due to the at most countability of \(\Aut (\pi)\), there exists \(\tau \in \Aut (\pi)\) such that \(A_\tau\) is non-ne\-gli\-gible. For this $\tau$, combining the equality \(\pi \compose (\tau \compose \lifting{u}) = \pi \compose \lifting{u} = \pi \compose \lifting{v}\) on \(\manifold{M}\)
with the fact that $\tau\circ \lifting{u}\in W^{s,p} (\manifold{M}, \lifting{\manifold{N}})$ and 
with the previous proposition, we obtain  $\lifting{v}=\tau\circ\lifting{u}$ a.e.\ in \(\manifold{M}\).
\end{proof}

\begin{corollary}
\label{db1}
Assume \eqref{dc1}--\eqref{dc5} and,  in addition, that   $\inj (\manifold{N})>0$ and that $\pi$ is a normal covering.\\
Let $\lifting{u}\in W^{s,p}_{loc}(\manifold{M} , \lifting{\manifold{N}})\setminus W^{s,p} (\manifold{M} , \lifting{\manifold{N}})$ and set $u:=\pi\compose \lifting{u}$. Then $u$ has no lifting $\lifting{v}\in W^{s,p}(\manifold{M}, \lifting{\manifold{N}})$.
\end{corollary}
\begin{proof}
  Argue by contradiction. By Proposition~\ref{proposition_uniqueness_normal}, there exists some $\tau\in \Aut(\pi)$ such that $\lifting{u}=\tau^{-1}\circ\lifting{v}$ a.e.\ on \(\manifold{M}\). 
  This leads to the contradiction $\lifting{u}\in W^{s,p} (\manifold{M}, \lifting{\manifold{N}})$.
\end{proof}

We now turn to uniqueness results involving solely the assumptions \eqref{dc1}--\eqref{dc5}.
\begin{proposition}
\label{de1}
Assume \eqref{dc1}--\eqref{dc5}.\\
Let \(\lifting{u}, \lifting{v} \in W^{s, p}_{loc} (\manifold{M}, \lifting{\manifold{N}})\) be such that
\(\pi \compose \lifting{u} = \pi \compose \lifting{v}\)  on \(\manifold{M}\). Assume, moreover, that   $\lifting{u}$ is continuous.\\
Then either \(\lifting{u} = \lifting{v}\) a.e.\ on \(\manifold{M}\) or \(\lifting{u} \ne \lifting{v}\) a.e.\ on \(\manifold{M}\).
\end{proposition}
\begin{proof}
  Assume that the set $C:=\{ y\in\manifold{M};\, \lifting{u}(y)= \lifting{v}(y)\}$ is non-negligible.  
  By continuity of \(\lifting{u}\), for each $x\in\manifold{M}$, there exist $\varepsilon=\varepsilon(x)>0$ and $r=r(x)>0$ such that $(\pi\compose\lifting{u} )\, (B_\varepsilon(x))$ is contained in an evenly covered geodesic  ball $U=U(x)$ of radius $r$.
   We consider the set
\begin{equation*}
D:=\{ x\in\manifold{M};\, C\cap B_\varepsilon(x)\text{ is non-negligible}\}.
\end{equation*}
By the assumption on $C$, the set $D$ is non-empty. 
We claim that 
\begin{equation*}
x\in D\implies [\text{the set $ B_\varepsilon(x) \setminus C$ is negligible}]. 
\end{equation*}
This claim clearly implies that the set $D$ is both open and closed, and thus, by connectedness, that $D=\manifold{M}$, whence (via the claim) the conclusion of the proposition.
It therefore remains to establish the claim. 

Let $x\in D$. Write $\pi^{-1}(U(x))$ as a disjoint union,  $\pi^{-1}(U(x))=\bigcup_{i\in I}V_i$, with $\pi:V_i\to U(x)$ a diffeomorphism. Since $\lifting{u}$ is continuous, there exists some $j\in I$ such that $\lifting{u}(B_{\varepsilon} (x))\subset V_j$. Let $\varphi(t):=\min\{ t/r, 1\}$, $\forall\, t\ge 0$, and set $f(y):=\varphi(d_{\lifting{N}}(\lifting{u}(y), \lifting{v}(y)))$, $\forall\, y\in B_\varepsilon (x)$. As in the proof of Proposition~\ref{proposition_uniqueness}, we have $f\in W^{s,p}(B_\varepsilon(x), \{ 0, 1\})$, and thus $f$ is constant. Since the set $f^{-1} (\{0\})$ is non-ne\-gli\-gible (by definition of the set $D$), we find that $f=0$ a.e.\ on \(B_{\varepsilon} (x)\), and thus $\lifting{u}=\lifting{v}$ a.e.\ in $B_\varepsilon (x)$, as claimed.
 \end{proof}

In the spirit of Corollary~\ref{db1}, we have the following consequence of Proposition~\ref{de1}.
\begin{corollary}
\label{de2}
Assume \eqref{dc1}--\eqref{dc5}. \\
Let $\lifting{u}\in W^{s,p}_{loc}(\manifold{M} , \lifting{\manifold{N}})\setminus W^{s,p} (\manifold{M}, \lifting{\manifold{N}})$ be a continuous map and set $u:=\pi\compose \lifting{u}$.\\
 If $u$ has a lifting $\lifting{v}\in W^{s,p}(\manifold{M}, \lifting{\manifold{N}})$, then $\lifting{u}\neq\lifting{v}$ a.e.
\end{corollary}
\section{Analytical singularity}
\label{s4}
In this section, we prove Theorem~\ref{b5}.  
In what follows, we assume that 
\begin{gather}
\label{f1}
0<s<1,\ p=1/s,
\\
\label{f2}
m\ge 2.
\end{gather}

\subsection{The basic ingredient}
\label{s4.1}
We start by proving  the existence of smooth maps $\lifting{u}:\Rset^m\to\lifting{\manifold{N}}$ such that $|\lifting{u}|_{W^{s,p}(B_1)}$ is arbitrarily large, while $|\pi\compose \lifting{u}|_{W^{s,p}(\Rset^m)}$ is arbitrarily small.

\begin{lemma}
\label{lemma_individual_singularity}
Assume \eqref{f1}--\eqref{f2}.  Let $r>0$ and $x_0\in\Rset^m$.\\
Let \(\pi \in C^\infty ({\lifting{\manifold{N}}}, \manifold{N})\) be a Riemannian covering, with $\lifting{\manifold{N}}$ connected.\\
Given \(\lifting{b}, \lifting{b}' \in {\lifting{\manifold{N}}}\) such that \(\lifting{b} \neq \lifting{b}'\) but \(\pi (\lifting{b}) = \pi (\lifting{b}')\),  and given $\varepsilon,  M>0$, there exists some $\lifting{u}\in C^\infty(\Rset^m, {\lifting{\manifold{N}}})$ such that
\begin{enumerate}
 \renewcommand{\labelenumi}{{\rm{(\roman{enumi})}}}
\renewcommand{\theenumi}{\roman{enumi}}
  \item \label{it_Ogh7Oo_1}
$\lifting{u}(x)=\lifting{b}$ when $|x-x_0|\ge r$,
\item \label{it_Ogh7Oo_2}
$\lifting{u}(x)=\lifting{b}'$ near $x_0$,
\item \label{it_Ogh7Oo_3}
$|\lifting{u}|_{W^{s,p}(B_r(x_0))}>M$,
\item \label{it_Ogh7Oo_4}
$|\pi\compose\lifting{u}|_{W^{s,p}(\Rset^m)}<\varepsilon$.
\end{enumerate}
\end{lemma}

\begin{proof}
With no loss of generality, we let $x_0=0$ and $r=1$. 

Assume that we are able to prove the lemma for \emph{some fixed} $\varepsilon_0$ and \emph{every} $M>0$.  Let $0<\varepsilon<\varepsilon_0$. 
Let $\lifting{u}$ as above, corresponding to $\varepsilon_0$ and to $M':=(\varepsilon M)/\varepsilon_0$. We define $\lambda>1$ by the equation $\lambda^{m-1}=\varepsilon_0/\varepsilon$, and we set $\lifting{v}(x):=\lifting{u}(\lambda x)$, $\forall\, x\in\Rset^m$. By scaling, $\lifting{v}$ satisfies items \eqref{it_Ogh7Oo_1}--\eqref{it_Ogh7Oo_4} (for $\varepsilon$ and $M$). It therefore suffices to establish the existence of $\lifting{u}$ satisfying \eqref{it_Ogh7Oo_1}--\eqref{it_Ogh7Oo_4} for \emph{some} $\varepsilon_0>0$ and \emph{arbitrary} $M>0$.

Since the manifold \({\lifting{\manifold{N}}}\) is connected, there exists a map \(\gamma \in C^\infty (\Rset, {\lifting{\manifold{N}}})\) such that \(\gamma (t) = \lifting{b}\) if \(t \le 0\) and \(\gamma (t) = \lifting{b}'\) if \(t \ge 1\). 
We define,  for every \(\delta\in (0,1)\),  the map \(\lifting{u}_\delta\in C^\infty(\Rset^m,   {\lifting{\manifold{N}}})\) through the formula
\[
 \lifting{u}_\delta(x) = \gamma \left(\frac{1 - 2 \abs{x}}{\delta}\right),\ \forall\, x\in\Rset^m.
\]

Clearly, $\lifting{u}_\delta$ satisfies \eqref{it_Ogh7Oo_1} and \eqref{it_Ogh7Oo_2}. In view of the above discussion, in order to complete the proof of the lemma it suffices to prove that
\begin{gather}
\label{ga1}
\lim_{\delta\to 0}|\lifting{u}_\delta|_{W^{s,p}(B_1)}=\infty,
\\
\label{ga2}
\limsup_{\delta\to 0}|\pi\circ\lifting{u}_\delta|_{W^{s,p}(\Rset^m)}<\infty.
\end{gather}
We note that 
\begin{equation*}
  \lim_{\delta\to 0 }\lifting{u}_\delta = \lifting{u} \ \text{a.e.\ in \(\Rset^m\)},
\end{equation*}
where 
\[
 \lifting{u} (x):=
 \begin{cases}
    \lifting{b}',& \text{if \(x \in {B_{1/2}}\)}\\
    \lifting{b},& \text{if \(x \in {\Rset^m\setminus B_{1/2}}\)}
 \end{cases},
\]
and that $\lifting{u}\not\in W^{s,p}(B_1, \lifting{\manifold{N}})$ 
(see the proof of Proposition~\ref{proposition_uniqueness}). This implies \eqref{ga1}.

In order to prove \eqref{ga2}, we set $u_\delta:=\pi\compose \lifting{u}_\delta$ and we note the following: 
\begin{gather}
\label{gb1}
u_\delta\equiv \pi(\lifting{b})\text{ in }\Rset^m\setminus U_\delta,\text{ where }U_\delta:=\{ x\in\Rset^m;\,  (1-\delta)/2<|x|< 1/2\},
\\
\label{gb2}
u_\delta\text{ is }\frac{C}{\delta}-\text{Lipschitz, with $C$ independent of $\delta$},
\\
\label{gb3}
d_{\manifold{N}}(u_\delta(x),u_\delta(y))\le C,\text{ with $C$ independent of $\delta$}.
\end{gather}
Combining \eqref{gb1}--\eqref{gb3}, we find (using the assumption $sp=1$) that\footnote{Here and in the sequel, $|U|$ denotes the Lebesgue measure of the set $U\subset\Rset^m$.}
\begin{equation*}
\begin{split}
|u_\delta|_{W^{s,p}(\Rset^m)}^p&
\lesssim \iint_{U_\delta\times\Rset^m}\frac{d_{\manifold{N}}(u_\delta(x),u_\delta(y))^p}{|x-y|^{m+1}}\dif x \dif y
\\
& \lesssim \iint_{x\in U_\delta, |x-y|\le\delta}\frac{|x-y|^p/\delta^p}{|x-y|^{m+1}}\dif x \dif y+\iint_{x\in U_\delta, |x-y|>\delta}\frac{1}{|x-y|^{m+1}}\dif x \dif y\\
&\lesssim \frac 1\delta\int_{U_\delta}\dif x=\frac 1\delta |U_\delta|\lesssim 1,
\end{split}
\end{equation*}
whence \eqref{ga2}.

The proof of Lemma~\ref{lemma_individual_singularity} is complete.
\end{proof}

\subsection{The analytic obstruction}
\label{s4.2}
Using Lemma~\ref{lemma_individual_singularity}, we construct an analytic singularity adapted to the case of the universal covering.

\begin{lemma}
\label{lemma_sing_compos}
Assume \eqref{f1}--\eqref{f2}.\\
Let \(\pi \in C^\infty ({\lifting{\manifold{N}}}, \manifold{N})\) be a non-trivial Riemannian covering, with $\lifting{\manifold{N}}$ connected. \\
Let $\manifold{M}\subset\Rset^m$ be a connected open set and let $a\in\overline{\manifold{M}}$. \\
Then
there exists a map \(\lifting{u} : \Rset^m \to \lifting{\manifold{N}}\) such that
\begin{enumerate}
 \renewcommand{\labelenumi}{{\rm{(\roman{enumi})}}}
\renewcommand{\theenumi}{\roman{enumi}}
  \item \label{it_Rosh8a_1} \(\lifting{u}\in C^\infty (\Rset^m \setminus \{a\}, \lifting{\manifold{N}})\),
  \item \label{it_Rosh8a_2}
$\lifting{u}\not\in W^{s,p}(\manifold{M}, \lifting{\manifold{N}})$,
\item \label{it_Rosh8a_3}
$\pi \compose \lifting{u}\in {W^{s, p}(\Rset^m, \manifold{N})}$,
\item \label{it_Rosh8a_4} \(\pi \compose \lifting{u}\) is a strong limit in \(W^{s, p} (\Rset^m, \manifold{N})\) of maps in \(C^\infty (\Rset^m, \manifold{N})\).
\end{enumerate}
\end{lemma}

Before proceeding to the proof of the lemma, we explain the meaning of items \eqref{it_Rosh8a_2} and \eqref{it_Rosh8a_4}. In \eqref{it_Rosh8a_2}, the $W^{s,p}$ semi-norm involves the Euclidean distance in $\Rset^m$, not the geodesic distance on $\manifold{M}$. 
The meaning of item \eqref{it_Rosh8a_4} is the following. We embed $\manifold{N}$ into some $\Rset^\nu$. Then there exist a sequence  $(u^j)\subset C^\infty(\Rset^m, \manifold{N})$ and some $b\in\manifold{N}$ such that $u^j-b, \pi\circ\lifting{u}-b\in W^{s,p}(\manifold{M}, \Rset^\nu)$ and $u^j-u\to 0$ in $W^{s,p}(\Rset^m, \Rset^\nu)$ as $j\to\infty$.

\begin{proof}[Proof of Lemma \ref{lemma_sing_compos}]
Since $m\ge 2$ and $a\in\overline{\manifold{M}}$, there exists a sequence of closed  balls $(\overline{B}_{\rho_k}(a_k))_{k\ge 0}$ such that:
\begin{enumerate}
 \renewcommand{\labelenumi}{{\rm{(\alph{enumi})}}}
\renewcommand{\theenumi}{\alph{enumi}}
  \item \label{it_faeP2J_1}
$\overline{B}_{\rho_k}(a_k)\subset \manifold{M}\setminus\{a\}$, $\forall\, k$,
\item \label{it_faeP2J_2}
the balls are mutually disjoint,
\item \label{it_faeP2J_3}
$a_k\to a$ (and thus $\rho_k\to 0$) as $k\to\infty$,
\item \label{it_faeP2J_4}
there exists a sequence $r_j\searrow 0$ such that $\{ x\in\Rset^m;\, |x-a|=r_j\}\cap \overline{B}_{\rho_k}(a_k)=\emptyset$, $\forall\, j$, $\forall\, k$.
\end{enumerate}

Since, by assumption, the cover \(\pi\) is non-trivial, there exist $\lifting{b}$ and $\lifting{b}'$ as in Lemma~\ref{lemma_individual_singularity}.
Let $(\varepsilon_k)_{k\ge 0}$ be a sequence of positive numbers to be defined later. Let, for every $k \ge 0$, $\lifting{u}_k$ be the map corresponding, as in Lemma~\ref{lemma_individual_singularity},  to $B_{\rho_k}(a_k)$, $\varepsilon_k$ and $M:=k+1$. We set,  for each \(x \in \Rset^m\), 
\[
 \lifting{u} (x):=
 \begin{cases}
  \lifting{u}_k (x),& \text{if \(x \in B_{\rho_k} (a_k)\) for some \(k \ge 0\)}\\
  \lifting{b}, & \text{otherwise}
 \end{cases}.
\]

Clearly, \eqref{it_Rosh8a_1} holds. Also clearly,
\begin{equation*}
|\lifting{u}|_{W^{s,p}(\manifold{M})}^p\ge |\lifting{u}|_{W^{s,p}(B_{\rho_k}(a_k))}^p\ge k+1,\ \forall\, k\ge 0,
\end{equation*}
and thus assertion \eqref{it_Rosh8a_2} holds.
By the countable patching property of Sobolev maps
\cite[Lemma 2.3]{Monteil_Van_Schaftingen}, we have (using the assumption \(0<s<1\)),
\begin{equation}
\label{eq_sing_compos_upper_bound_series}
    |\pi \compose \lifting{u}|_{W^{s, p} (\Rset^m)}^p
\le 
    2^p 
    \sum_{k \ge 0} \rho_k^{m - 1}
    |\pi \compose \lifting{u}_k|_{W^{s, p} (\Rset^m)}^p<2^p\sum_{k\ge 0}\rho_k^{m - 1}\varepsilon_k.
    \end{equation}
    We now choose $\varepsilon_k$ such that $\sum_{k\ge 0}\rho_k^{m - 1}\varepsilon_k<\infty$ and obtain \eqref{it_Rosh8a_3}.
    
    Finally, it remains to prove item \eqref{it_Rosh8a_4}. For \emph{scalar} functions, this follows from \eqref{it_Rosh8a_3}, but some care is needed for manifold-valued maps.  With $r_j$ as in \eqref{it_faeP2J_4}, set $u:=\pi\compose\lifting{u}:\Rset^m\to\manifold{N}$, $b:=\pi(\lifting{b})$  and define 
    \begin{equation*}
  u^j(x):=\begin{cases}
  u(x),&\text{if }|x-a|\ge r_j\\
  b,&\text{if }|x-a|<r_j
  \end{cases}.
    \end{equation*}
    Clearly, $u^j\in C^\infty(\Rset^m, \manifold{N})$, $u^j-b, u-b\in W^{s,p}(\Rset^m, \Rset^\nu)$ and $u^j- u\to 0$ a.e.\ and in $L^p(\Rset^m)$ as $j\to \infty$. It thus suffices to prove that $|u-u^j|_{W^{s,p}(\Rset^m)}\to 0$ as $j\to\infty$. For this purpose, we note that $|u-u^j|_{W^{s,p}(\Rset^m)}=|v^j|_{W^{s,p}(\Rset^m)}$, where
    \begin{equation}
    \label{gc1}
    v^j:=u-u^j+b=\begin{cases}
    \pi\compose \lifting{u}_k,&\text{in } B_{\rho_k}(a_k),\text{ if }B_{\rho_k}(a_k)\subset B_{r_j}(a)
    \\
    b,&\text{elsewhere}
    \end{cases}.
    \end{equation}
  By \eqref{eq_sing_compos_upper_bound_series} and the choice of $\varepsilon_k$,  we have $|v^j|_{W^{s,p}(\Rset^m)}\to 0$ as $j\to\infty$. 
  
  The proof of Lemma~\ref{lemma_sing_compos} is complete. 
  \end{proof} 

\noindent
\emph{Proof of Theorem~\ref{b5} for the universal covering of connected, non-simply-connected, compact Riemannian manifolds $\manifold{N}$.} When $\manifold{M}$ is a smooth bounded open set in $\Rset^m$, we first note that, on $\manifold{M}\times\manifold{M}$, the geodesic distance $d_{\manifold{M}}$ is equivalent to the Euclidean distance in $\Rset^m$. It then suffices to
combine Lemma~\ref{lemma_sing_compos} with Corollary~\ref{db1} (applied in the connected set $\manifold{M}\setminus\{a\}$). The case of a manifold reduces to this special case, since the  analytical singularity constructed in Lemma \ref{lemma_sing_compos}
is constant outside  an arbitrarily small neighborhood of $a$. 
\qed

\subsection{A variant of the analytic obstruction}
\label{s4.3}

In the general case, Theorem~\ref{b5} can be obtained via a suitable variant of Lemma~\ref{lemma_sing_compos}.

\begin{lemma}
\label{p2}
Assume \eqref{f1}--\eqref{f2}.\\
Let \(\pi \in C^\infty ({\lifting{\manifold{N}}}, \manifold{N})\) be a non-trivial Riemannian covering, with $\lifting{\manifold{N}}$ connected. \\
Let $b\in \manifold{N}$ and write $\pi^{-1}(\{b\})=\{ \lifting{b}_i;\, i\in I\}$. \\
Let $\manifold{M}\subset\Rset^m$ be a connected  open set and let $a\in\overline{\manifold{M}}$. \\
Then there exist a family $(U_i)_{i\in I}$ of open sets and a family $(\lifting{u}_i)_{i\in I}$ of maps such that
\begin{enumerate}
 \renewcommand{\labelenumi}{{\rm{(\roman{enumi})}}}
\renewcommand{\theenumi}{\roman{enumi}}
  \item \label{it_Ix4zi2_1}
$U_i\subsetneqq\manifold{M}\setminus\{a\}$, $\forall\, i\in I$,
\item \label{it_Ix4zi2_2}
$\displaystyle U_i\cap \overline{\textstyle\bigcup_{j\neq i}U_j}=\emptyset$,  $\forall\, i\in I$,
\item \label{it_Ix4zi2_3}
$\displaystyle\manifold{M}\setminus \overline{\textstyle\bigcup_{j\neq i}U_j}$ is connected, $\forall\, i\in I$,
\item \label{it_Ix4zi2_4}
$\lifting{u}_i\in C^\infty(\Rset^m\setminus\{a\}, \lifting{\manifold{N}})$,  $\forall\, i\in I$,
\item \label{it_Ix4zi2_5}
$\lifting{u}_i\equiv\lifting{b}_i$ in $\Rset^m\setminus (U_i\cup\{a\})$, 
  $\forall\, i\in I$,
  \item \label{it_Ix4zi2_6}
  $\lifting{u}_i\not\in W^{s,p}(U_i, \lifting{N})$,  $\forall\, i\in I$,
  \item \label{it_Ix4zi2_7}
  if we set \[u:=\begin{cases}
  \pi\compose\lifting{u}_i,&\text{in }U_i\\
  b,&\text{in }\Rset^m\setminus\bigcup_{i\in I}U_i
  \end{cases},\] 
  then $u\in C^\infty(\Rset^m\setminus\{a\}, {\manifold{N}})$ and $u\in W^{s,p}(\Rset^m, {\manifold{N}})$,
  \item \label{it_Ix4zi2_8}
$u$ is the strong limit in \(W^{s, p} (\Rset^m, \manifold{N})\) of maps in \(C^\infty (\Rset^m, \manifold{N})\).
\end{enumerate}
\end{lemma}

\begin{proof}
Our construction is again based on a family of balls, but this time indexed over $k\ge 0$ and $i\in I$ (we recall that the set \(I\) is at most countable). The requirements on the closed  balls $(\overline{B}_{\rho_{k,i}}(a_{k,i}))_{k\ge 0, i\in I}$ are the following:
\begin{enumerate}[(a)]
  \item \label{it_aeW4Ei_1}
$\overline{B}_{\rho_{k,i}}(a_{k,i})\subset \manifold{M}\setminus\{a\}$, $\forall\, k$, $\forall\, i$,
\item \label{it_aeW4Ei_2}
the  balls are mutually disjoint,
\item \label{it_aeW4Ei_3}
$a_{k,i}\to a$ (and thus $\rho_{k,i}\to 0$) as $k + i\to\infty$,
\item \label{it_aeW4Ei_4}
there exists a sequence $r_j\searrow 0$ such that $\{ x\in\Rset^m;\, |x-a|=r_j\}\cap \overline{B}_{\rho_{k,i}}(a_{k,i})=\emptyset$, $\forall\, j$, $\forall\, k$, $\forall\, i$.
\end{enumerate}
Set $U_i:=\bigcup_{k\ge 0}{B}_{\rho_{k,i}}(a_{k,i})$. 
Clearly, $\overline{U_i}=\bigcup_{k\ge 0}\overline{B}_{\rho_{k,i}}(a_{k,i})\cup\{ a\}$, and 
\eqref{it_Ix4zi2_1} and \eqref{it_Ix4zi2_2} hold.
By a straightforward argument, assumptions  \eqref{it_aeW4Ei_2} and \eqref{it_aeW4Ei_3}, combined with the fact that $\manifold{M}$ is connected and \(m \ge 2\), imply \eqref{it_Ix4zi2_3}. (Actually, we have the more general property that $\manifold{M}\setminus \overline{\bigcup_{j\in J}U_j}$ is connected, $\forall\, J\subseteq I$.) 
 
We next define $\lifting{u}_i$, $i\in I$. Since the covering \(\pi\) is non-trivial, we can consider, for each $i$, some $j=j(i)\in I\setminus\{i\}$. Let, for every $k$, $\lifting{u}_{k, i}$ correspond,   as in Lemma~\ref{lemma_individual_singularity},  to $\lifting{b}:=\lifting{b}_{i}$, $\lifting{b}':= \lifting{b}_{j}$, to the ball $B_{\rho_{k,i}}(a_{k,i})$, and to the numbers $\varepsilon_{k,i}$ and $M:=k+1$. By analogy with the proof of Lemma~\ref{lemma_sing_compos}, we require that $\sum_{k\ge 0, i\in I}\rho_{k, i}^{m-1}\varepsilon_{k,i}<\infty$. 
  We set
\[
 \lifting{u}_i (x):=
 \begin{cases}
  \lifting{u}_{k,i} (x),& \text{if \(x \in B_{\rho_{k,i}} (a_{k,i})\) for some \(k \ge 0\)}\\
  \lifting{b}_i, & \text{otherwise}
 \end{cases}.
\]
Following the proof of Lemma~\ref{lemma_sing_compos}, we find that \eqref{it_Ix4zi2_4}  through \eqref{it_Ix4zi2_8} hold. 

  The proof of Lemma~\ref{p2} is complete. 
\end{proof}

\begin{proof}[Proof of Theorem~\ref{b5} in the general case] Again, we may assume that $\manifold{M}$ is an open set in $\Rset^m$. Let $u$ be as in Lemma~\ref{p2}. Argue by contradiction and assume that $u=\pi\compose\lifting{u}$ for some $\lifting{u}\in W^{s,p}(\manifold{M}, \lifting{\manifold{N}})$. 
Let $i\in I$. 
By Corollary~\ref{de2} applied to $u$ in the connected open set $\manifold{M}\setminus\overline{\bigcup_{j\neq i}U_j}$, for the smooth lifting $\lifting{u}_i$, we have $\lifting{u}\neq \lifting{u}_i$ a.e.\ in the set $V:=\manifold{M}\setminus \overline{\bigcup_{j\in I}U_j}$. Thus, a.e.\ in $V$, we have $\lifting{u}(x)\not\in \{ \lifting{b}_i;\, i\in I\}$. This contradicts the facts that $V$ has positive measure and $\pi\circ\lifting{u}(x)=b$, $\forall\, x\in V$.
\end{proof}

\end{document}